\documentclass{amsart}

\usepackage{amsfonts}
\usepackage{amssymb}
\usepackage{mathtools}
\usepackage{tikz}
\usetikzlibrary{arrows,decorations.pathreplacing}
\usetikzlibrary{shapes}
\usetikzlibrary{arrows,calc}

\usepackage[
	pdftitle={},
	pdfauthor={},
	ocgcolorlinks,
	linkcolor=linkblue,
	citecolor=linkred,
	urlcolor=linkred]
{hyperref}
\usepackage{color}
\definecolor{linkred}{rgb}{0.75,0,0}
\definecolor{linkblue}{rgb}{0,0,0.75}
\usepackage{multicol}
\usepackage{multirow}
\usepackage{enumerate}

\theoremstyle{plain}
\newtheorem{theorem}{Theorem}
\newtheorem{deflem}{Definition/Lemma}
\newtheorem{proposition}{Proposition}[section]
\newtheorem{lemma}[proposition]{Lemma}
\newtheorem{conjecture}{Conjecture}
\newtheorem{corollary}[proposition]{Corollary}

\newcommand{\bt}{\begin{theorem}}
\newcommand{\et}{\end{theorem}}

\theoremstyle{definition}
\newtheorem{definition}[proposition]{Definition}
\newtheorem{remark}[proposition]{Remark}
\newcommand{\beq}{\begin{equation}}
\newcommand{\eeq}{\end{equation}}

\newcommand{\cal}{\mathcal}

\newcommand{\cc}{\mathcal{C}}

\newcommand{\ce}{\mathcal{E}}
\newcommand{\co}{\mathcal{O}}
\newcommand{\cf}{{\cal F}}

\newcommand{\bc}{\mathbb{C}}
\newcommand{\be}{\mathbb{E}}
\newcommand{\bj}{\mathbb{J}}
\newcommand{\bk}{\mathbb{K}}
\newcommand{\bl}{\mathbb{L}}
\newcommand{\bn}{\mathbb{N}}
\newcommand{\bp}{\mathbb{P}}
\newcommand{\bq}{\mathbb{Q}}

\newcommand{\bz}{\mathbb{Z}}

\newcommand{\modm}{\cal M}
\newcommand{\un}{1\!\!1}
\newcommand{\dtt}{\mathbf{t}}

\newcommand{\pd}[2]{\frac{\partial#1}{\partial#2}}
\newcommand{\ck}{\mathcal{K}}
\newcommand{\cw}{\mathcal{W}}

\begin{document}
	
\title[Polynomial relations among kappa classes]{Polynomial relations among kappa classes on the moduli space of curves}
\author{Maxim Kazarian}
\address{National Research University Higher School of Economics, Skolkovo Institute of Science and Technology}
\email{\href{mailto:kazarian@mccme.ru}{kazarian@mccme.ru}}
\author{Paul Norbury}
\address{School of Mathematics and Statistics, University of Melbourne, VIC 3010, Australia}
\email{\href{mailto:norbury@unimelb.edu.au}{norbury@unimelb.edu.au}}
\thanks{}
\subjclass[2020]{14H10; 14D23; 32G15}
\date{\today}

\begin{abstract}
We construct an infinite collection of universal---independent of $(g,n)$---polynomials in the Miller-Morita-Mumford classes $\kappa_m\in H^{2m}( \overline{\cal M}_{g,n},\bq)$, defined over the moduli space of genus $g$ stable curves with $n$ labeled points.  We conjecture vanishing of these polynomials in a range depending on $g$ and $n$.
\end{abstract}

\maketitle

\tableofcontents

\section{Introduction}

Let $\overline{\modm}_{g,n}$ be the moduli space 
of genus $g$ stable curves---curves with only nodal singularities and finite automorphism group---with $n$ labeled points disjoint from nodes, for any $g,n\in\bn=\{0,1,2,...\}$ satisfying $2g-2+n>0$.   Define the forgetful map $\pi:\overline{\modm}_{g,n+1}\to\overline{\modm}_{g,n}$ that forgets the last point, and define $\omega_\pi$ to be the relative dualising sheaf of the map $\pi$.  There are natural sections $p_i:\overline{\modm}_{g,n}\to\overline{\modm}_{g,n+1}$ of the map $\pi$ for each $i=1,...,n$.   Define $L_i= p_i^*\omega_\pi$ to be the line bundle $L_i\to\overline{\modm}_{g,n}$ with fibre above $[(C,p_1,\ldots,p_n)]$ given by $T_{p_i}^*C$ and $\psi_i=c_1(L_i)\in H^{2}(\overline{\modm}_{g,n},\bq) $ to be its first Chern class.  Define the Hodge bundle $\be_g=\pi_*\omega_\pi$ which is a rank $g$  bundle over $\overline{\modm}_{g,n}$ and $\lambda_i=c_i(\be_g)\in H^{2i}(\overline{\modm}_{g,n},\bq) $ to be its $i$th Chern class for $i=1,...,g$.  Define the Miller-Morita-Mumford classes $\kappa_m=\pi_*\psi^{m+1}_{n+1}\in H^{2m}( \overline{\cal M}_{g,n},\bq)$ where $\pi$ is the forgetful map $\overline{\modm}_{g,n+1}\stackrel{\pi}{\longrightarrow}\overline{\modm}_{g,n}$.  We will refer to $\psi_i$ and $\kappa_m$ as $\psi$ classes, respectively $\kappa$ classes.

\subsection{The construction}\label{sec:constr}
Consider $\{s_i\in\bq\mid i>0\}$ defined by
$$\exp\left(-\sum_{i>0} s_it^i\right)=\sum_{k=0}^\infty(-1)^k(2k+1)!!t^k
$$
and $\{\sigma_i\in\bq\mid i>0\}$ defined by
$$\exp\left(-\sum_{i>0} \sigma_it^i\right)=\sum_{k=0}^\infty(-1)^kk!t^k.
$$
Define
$$\bk=K_0+K_1+K_2+...=\exp(\sum s_i\kappa_i)\in H^*(\overline{\modm}_{g,n},\bq)$$
and
$$\bj=J_0+J_1+J_2+...=\exp(\sum \sigma_i\kappa_i)\in H^*(\overline{\modm}_{g,n},\bq)$$
where $K_m,J_m\in H^{2m}(\overline{\modm}_{g,n},\bq)$.
This defines two sequences of homogeneous polynomials in $\kappa_m$.  The first few polynomials are given by:
$$K_1=3\kappa_1,\quad K_2=\frac32(3\kappa_1^2-7\kappa_2),\quad K_3=\frac32(3\kappa_1^3-21\kappa_1\kappa_2+46\kappa_3)$$
$$K_4=\frac98(3\kappa_1^4-42\kappa_1^2\kappa_2+49\kappa_2^2+184\kappa_1\kappa_3-562\kappa_4)
$$
and
$$J_1=\kappa_1,\quad J_2=\frac12(\kappa_1^2-3\kappa_2),\quad J_3=\frac{1}{6}(\kappa_1^3-9\kappa_1\kappa_2+26\kappa_3)$$
$$J_4=\frac{1}{24}(\kappa_1^4-18\kappa_1^2\kappa_2+27\kappa_2^2+104\kappa_1\kappa_3-426\kappa_4).
$$
\begin{conjecture}  \label{conj}
The following polynomial relations among the $\kappa_j$ hold:
\begin{align}
K_m(\kappa_1,...,\kappa_m)&=0\quad\text{for}\quad m>2g-2+n\quad\text{except}\quad (m,n)=(3g-3,0); \label{Krel}\\
J_m(\kappa_1,...,\kappa_m)&=0\quad\text{for}\quad \left\{\begin{array}{ll}m>2g-2+n&\\
m=2g-2+n,&n>1.\end{array}\right.\label{Jrel}
\end{align}
\end{conjecture}
\noindent The conjecture is trivially true when $m>3g-3+n$, and hence says nothing about $g=0$.  Evidence for the conjecture is given by the following theorem.   

\begin{theorem}  \label{numeric}
For any polynomial $P(\psi_1,...,\psi_n)\in H^{2m}(\overline{\mathcal{M}}_{g,n})$ with $n>0$ and $m<g-1$,
\[
\int_{\overline{\modm}_{g,n}}\bk\cdot P(\psi_1,...\psi_n)=0.
\]
\end{theorem}
Theorem~\ref{numeric} is the first step towards a proof of a weak version of Conjecture~\ref{conj} involving the tautological ring $RH^m(\overline{\modm}_{g,n})\subset H^{2m}(\overline{\modm}_{g,n},\bq)$, \cite{FPaCon}---see Corollary~\ref{regweak} in Section~\ref{sectcf}.  An analogous statement for $\int_{\overline{\modm}_{g,n}}\bj\cdot P(\psi_1,...\psi_n)$ is reduced to a purely combinatorial statement, without reference to the moduli space of stable curves, in Section~\ref{Jweak}.  

The construction of the polynomials in $\kappa$ classes $J_m$ and $K_m$ resembles the following construction of polynomials in $\kappa$ classes used to produce polynomial relations among $\kappa$ classes over the moduli space of curves of compact type $\modm_{g,n}^{\text{ct}}$.  A stable curve is of compact type if its dual graph is a tree, and we have:
\[\modm_{g,n}\subset\modm_{g,n}^{\text{ct}}\subset\overline{\modm}_{g,n}.
\]
In \cite{PanKap}, Pandharipande found relations among $\kappa$ classes over the moduli space of curves of compact type $\modm_{g,n}^c$.  Pixton \cite{PixTau} found explicit formulae for these relations in a form strikingly similar to the construction of $\bj$ and $\bk$.  Define $\{p_i\in\bq\mid i>0\}$ by
$$\exp\left(-\sum_{i>0}p_it^i\right)=\sum_{k=0}^\infty(-1)^k(2k-1)!!t^k
$$
where $(-1)!!=1$.  Define
\[\bp=P_0+P_1+P_2+...=\exp(\sum p_i\kappa_i)\in H^*(\modm_{g,n}^{\text{ct}},\bq)\]
where $P_m\in H^{2m}(\overline{\modm}_{g,n},\bq)$.
\begin{theorem}[Pixton 2013]   \label{pixton}
\[P_m=0\in H^{2m}(\modm_{g,n}^{\text{ct}},\bq),\quad 2m>2g-2+n.\]
\end{theorem}
Theorem~\ref{pixton} appears in \cite[Section 3.4]{PixTau}, see also \cite[Section 3.5]{PanCal}, as part of a more general collection of relations.

Due to the natural restriction properties of the classes $\bk$ and $\bj$ to lower dimensional strata in the moduli space of stable curves, the conjectural vanishing of the polynomials $K_m$ and $J_m$ can be reduced to properties of the (non-vanishing) polynomials $K_{2g-2+n}$ and $J_{2g-2+n}$.  This is proven in Propositions~\ref{JHodge} and \ref{Kpsi} and motivates Conjectures~\ref{conjhodge}, \ref{eigenfunction} and \ref{KTheta} below.   

The following two conjectures seek to characterise the polynomials $J_{2g-2+n}$ and $K_{2g-2+n}$ in such a way to give a possible mechanism for the proof of Conjecture~\ref{conj}.   As usual $\kappa_i$, hence also $J_m$ and $K_m$, represents a class in $H^*(\overline{\mathcal{M}}_{g,n},\bq)$ for any $(g,n)$.  If we write the subscript $m=m(g,n)$ as a function of $g$ and $n$, we have fixed $g$ and $n$.  In particular, $J_{2g-2+n},K_{2g-2+n}\in H^*(\overline{\mathcal{M}}_{g,n},\bq)$.
\begin{conjecture}  \label{conjhodge}
$$J_{2g-2+n}=\left\{\begin{array}{cc}(-1)^{g}\lambda_{g-2}\lambda_g\in H^*(\overline{\mathcal{M}}_{g},\bq)&n=0\\(-1)^{g-1}\lambda_{g-1}\lambda_g\in H^*(\overline{\mathcal{M}}_{g,1},\bq)&n=1\end{array}\right.
$$
\end{conjecture}
Conjecture~\ref{conjhodge} has been checked up to $g=4$ using the Sage package admcycles \cite{DSZadm} which is based on Pixton's relations \cite{PPZRel,PixTau}.
\begin{conjecture}  \label{eigenfunction}
For $g>1$
$$K_{2g-2+n}=\left(\prod_{i=1}^n\psi_i\right)\pi^*K_{2g-2}$$
and $K_{2g-2}$ is the unique degree $2g-2$ polynomial up to scale with the property that all of its pullbacks times the product of $\psi$ classes are polynomials in $\kappa$ classes.
\end{conjecture}
Conjecture~\ref{eigenfunction} is compatible with the push-forward property
\[\pi_*K_{2g-2+n+1}=(2g-2+n)K_{2g-2+n}\]
proven in Section~\ref{sec:prop}.  Moreover, an easy consequence of the push-forward property is the following proven $g=1$ version of Conjecture~\ref{eigenfunction}.  For $K_1\in H^*(\overline{\modm}_{1,1},\bq)$,
\[K_{n}=\left(\prod_{i=2}^n\psi_i\right)\pi^*K_{1}\in H^*(\overline{\modm}_{1,n},\bq).
\]
The uniqueness part of Conjecture~\ref{eigenfunction} can be restated as follows.  For $g>1$, a homogeneous degree $2g-2$ polynomial $P(\kappa_1,...,\kappa_{2g-2})$ satisfies the property that $\left(\prod_{i=1}^n\psi_i\right)\pi^*P\in H^*(\overline{\modm}_{g,n},\bq)$ is a polynomial in $\kappa$ classes for each $n>0$ iff $P(\kappa_1,...,\kappa_{2g-2})$ is a scalar multiple of $K_{2g-2}$.

Conjectures~\ref{conjhodge} and \ref{eigenfunction} imply Conjecture~\ref{conj} via vanishing of well chosen restriction maps---see Section~\ref{sec:geom}, Propositions~\ref{JHodge} and~\ref{Kpsi}.  Thus, expressions for $\lambda_{g-2}\lambda_g$ and $\lambda_{g-1}\lambda_g$ in terms of known polynomials in the $\kappa$ classes, such as
\[ \lambda_2\lambda_3= J_5=\tfrac{1}{120}\kappa_1^5-\tfrac14\kappa_1^3\kappa_2+\tfrac98\kappa_1\kappa_2^2+\tfrac{13}{6}\kappa_1^2\kappa_3-\tfrac{13}{2}\kappa_2\kappa_3-\tfrac{71}{4}\kappa_1\kappa_4+\tfrac{461}{5}\kappa_5 \in H^*(\overline{\modm}_{3},\bq)\]
are enough to prove the vanishing $J$-polynomial relations among $\kappa$ classes.  Similarly, properties of the action of the operator $\psi_{n+1}\pi^*$ on the ring of $\kappa$ polynomials would be enough to prove the vanishing $K$-polynomial relations among $\kappa$ classes.

In \cite{NorNew} a collection of cohomology classes $\Theta_{g,n}\in H^{2(2g-2+n)}(\overline{\modm}_{g,n},\bq)$, conjecturally related to the supermoduli space of curves $\widehat{\modm}_{g,n}$, \cite{NorEnu,SWiJTG}, are defined.  We recall the construction in Section~\ref{spincohft}.
\begin{conjecture}  \label{KTheta}
For $2g-2+n>0$,
$$K_{2g-2+n}=\Theta_{g,n}.$$
\end{conjecture}
It is proven in \cite{NorNew} that for $n\geq 0$, $\Theta_{g,n+1}=\psi_{n+1}\pi^*\Theta_{g,n}$.  Hence Conjecture~\ref{KTheta} implies Conjecture~\ref{eigenfunction} and thus also the vanishing $K_m(\kappa_1,...,\kappa_m)=0$ for $m>2g-2+n$ except $(m,n)=(3g-3,0)$.  A summary of the implications is as follows:\\

\vspace{-.2cm}
\begin{center}
Conjecture~\ref{conjhodge} + Conjecture~\ref{eigenfunction} $\ \Rightarrow\ $ Conjecture~\ref{conj}

\vspace{.3cm}
Conjecture~\ref{KTheta} $\Rightarrow$ Conjecture~\ref{eigenfunction} $\Rightarrow$ \eqref{Krel} of Conjecture~\ref{conj}.
\end{center}
\vspace{.1cm}

We prove Theorem~\ref{numeric} in Section~\ref{sec:van} via regularity of Virasoro equations, or equivalently topological recursion, depending on a parameter.  We also show that a weak version of Conjecture~\ref{conj}, which considers vanishing intersection numbers, can be reduced to verifying the regularity of generating functions depending on a parameter and produced via  topological recursion,  following the proof of Theorem~\ref{numeric}.  We relate the classes $\bk$ to the Br\'ezin-Gross-Witten tau function and conjecturally to the classes $\Theta_{g,n}$ in Sections~\ref{sec:relbgw} and \ref{sec:geom}.

\noindent {\em Acknowledgements.}   PN would like to thank Rahul Pandharipande, Aaron Pixton, Johannes Schmitt and Mehdi Tavakol for useful conversations.  MK is supported by the National Research University Higher School of Economics.    PN was supported under the  Australian
Research Council {\sl Discovery Projects} funding scheme project number DP180103891.

\section{Vanishing of intersection numbers} \label{sec:van}
In this section we consider intersection numbers $\int_{\overline{\modm}_{g,n}}\bk\cdot\omega$ and $\int_{\overline{\modm}_{g,n}}\bj\cdot\omega$ for any tautological class $\omega\in RH^*(\overline{\modm}_{g,n})$.  We prove Theorem~\ref{numeric} regarding vanishing of intersection numbers of $\bk$ with polynomials in $\psi$ classes via regularity of a related potential.  Using this same idea, we prove that the more general vanishing of intersection numbers with $\omega\in RH^m(\overline{\mathcal{M}}_{g,n})$ for $m<g-1$, i.e. a weak version of Conjecture~\ref{conj}, is equivalent to a purely combinatorial problem stated in terms of regularity of correlators, without reference to the moduli space of curves.

\subsection{Topological recursion for the $\kappa$-potential}\label{sectcf}

Consider two sets of independent variables $s=(s_1,s_2,\dots)$ and $t=(t_0,t_1,\dots)$. Set
$$\ck=\exp\Bigl(\sum_{i=1}^\infty s_i\kappa_i\Bigr)$$
and introduce the function
$$
\cf(\hbar,t,s)=\sum_{g,n}\frac{\hbar^g}{n!}\sum_{k_1,\dots,k_n}\int_{\overline{\modm}_{g,n}}
\ck\;\psi_1^{k_1}\dots\psi_n^{k_n}\;t_{k_1}\dots t_{k_n}
$$
where the sum is over all $g,n\in\bn$ satisfying $2g-2+n>0$ and $(k_1,\dots,k_n)\in\bn^n$.
This is the generating function for intersection numbers of all possible monomials in $\kappa$ classes and $\psi$ classes. If we set $s_i=0$ for all $i$, this specializes to the Kontsevich-Witten potential enumerating intersection numbers of $\psi$-monomials,
$$\cf(\hbar,t,0)=F^{KW}(\hbar,t).$$

Define polynomials $h_i=h_i(s)$ by the formal expansion
\begin{equation}\label{hschange}
1+h_1t+h_2t^2+\dots=\exp(-\sum s_it^i).
\end{equation}

\begin{theorem}[\cite{MZoInv}]\label{KWshift}
The full potential $\cf(\hbar,t,s)$ is obtained from the Kontsevich-Witten potential $F^{KW}(\hbar,t)$ by a suitable shift of variables
$$\cf(\hbar,t,s)=F^{KW}(\hbar,t_0,t_1,t_2-h_1(s),t_3-h_2(s),\dots).$$
\end{theorem}

\begin{remark}
The Kontsevich-Witten potential is a power series in the variables $t_0,t_1,\hbar$ whose coefficients are polynomials in the remaining variables. Therefore, the Taylor expansion of $F^{KW}$ at the point $t=(0,0,-h_1,-h_2,\dots)$ is well defined.
\end{remark}

\begin{proof} 
There is a standard inductive procedure reducing computation of intersection numbers involving $\kappa$ classes to those involving just $\psi$ classes. This procedure is based on the following relations describing the behaviour of $\kappa$ classes under the forgetful map $\pi:\overline{\modm}_{g,n+1}\to\overline{\modm}_{g,n}$:
$$\kappa_m=\pi_*(\psi_{n+1}^{m+1}),\qquad \psi_{n+1}^a\pi^*(\kappa_m)=\psi_{n+1}^a(\kappa_m-\psi_{n+1}^m),~a>0.$$
These relations imply relations between intersection numbers which are equivalent to the following linear partial differential equations:
\begin{equation}\label{sdependence}
V_m\cf=0,\quad m=1,2,\dots,\qquad V_m=\pd{}{s_m}-\sum_{i=0}^\infty h_{i-m-1}(s)\pd{}{t_i}.
\end{equation}
These equations mean that the function $\cf$ is constant along the phase curves of the pairwise commuting vector fields $V_m$ and thus along the whole distribution generated by these fields. The integral subvariety of this distribution passing through the point $(t,s)$ intersects the subspace $s=0$ at the point $(q,0)$ where $q=(t_0,t_1,t_2-h_1(s),t_3-h_2(s),\dots)$. This implies the equality of the theorem.
\end{proof}
\begin{remark}
Equation \eqref{hschange} provides an invertible change of $h$ and $s$ coordinates. From now on, it will be more convenient for us to apply this change and to use the coordinates $h_1,h_2,\dots$ instead of $s_1,s_2,\dots$ in all considerations. So, for example, by $\ck(h)$ we mean the class
$$\ck(h)=\exp\bigl((-h_1)\,\kappa_1+(\frac12h_1^2-h_2)\,\kappa_2+(-\frac13h_1^3+h_1h_2-h_3)\,\kappa_3+\dots\bigr)$$
which is the exponential $e^{\sum s_i\kappa_i}$ expressed in terms of $h$ coordinates, and similarly for the $\kappa$-potential $\cf$. In $h$-coordinates, the equalities~\eqref{sdependence} take a simpler form
$$\pd{\cf}{h_m}+\pd{\cf}{t_{m+1}}=0,\quad m=1,2,\dots.$$
\end{remark}

\bigskip
\begin{corollary}\label{Vir}
The potential $\cf$ obeys the following Virasoro constraints.  For any $m\ge-1$ we have
\begin{multline*}
\frac12\sum_{i+j=m-1}(2i+1)!!(2j+1)!!
\left(\hbar\pd{^2\cf}{t_i\partial t_j}+\pd{\cf}{t_i}\pd{\cf}{t_j}\right)+\\
\sum_{k-i=m}(t_i-h_{i-1})\frac{(2k+1)!!}{(2i-1)!!}\pd{\cf}{t_k}+
\delta_{m,-1}\frac{t_0^2}{2}+\delta_{m,0}\frac{\hbar}{8}=0.
\end{multline*}
where we set $h_0=1$, $h_{-1}=0$.
\end{corollary}
\begin{proof}
This follows directly from the Virasoro constraints for the Kontsevich-Witten potential.
\end{proof}

Here is an equivalent form of the Virasoro constraints. Define the following partial differential operators:
$$L_m=\frac\hbar2\sum_{i+j=m-1}(2i+1)!!(2j+1)!!
\pd{^2}{t_i\partial t_j}+
\sum_{k-i=m}t_i\frac{(2k+1)!!}{(2i-1)!!}\pd{}{t_k}+
\delta_{m,-1}\frac{t_0^2}{2}+\delta_{m,0}\frac{\hbar}{8}.
$$
Then we have
$$\left(L_m-\sum_{k-i=m}h_{i-1}\frac{(2k+1)!!}{(2i-1)!!}\pd{}{t_k}\right)e^{\hbar^{-1}\cf}=0,
\qquad m=-1,0,1,2,\dots$$

\bigskip
There is another equivalent way to represent the Virasoro constraints, namely, in the form of \emph{topological recursion} \cite{EOrInv}. Let us expand $\cf=\sum_{g,n}\hbar^g F_{g,n}$ where $F_{g,n}$ is the contribution of intersection numbers on a single moduli space,
$$
F_{g,n}=\frac{1}{n!}\sum_{k_1,\dots,k_n}\int_{\overline{\modm}_{g,n}}
\ck(h)\;\psi_1^{k_1}\dots\psi_n^{k_n}\;t_{k_1}\dots t_{k_n}.
$$
There is a different way to collect these intersection numbers into a generating polynomial, namely, using the so called \emph{correlator differentials}
$$w_{g,n}(z_1,\dots,z_n)=\sum_{k_1,\dots,k_n}\int_{\overline{\modm}_{g,n}}
\ck(h)\;\psi_1^{k_1}\dots\psi_n^{k_n}~\prod_{i=1}^n\frac{(2k_i+1)!!dz_i}{z^{2k_i+2}_i}.$$
Equivalently, introduce the \emph{loop insertion operators}
$$\delta_i=\sum_{k=0}\frac{(2k+1)!!dz_i}{z^{2k+2}_i}\pd{}{t_k}.$$
Then we have
$$w_{g,n}=\delta_1\dots\delta_nF_{g,n}.$$

\begin{theorem}\label{thTR}
The differentials $w_{g,n}$ satisfy the relations of topological recursion for the spectral curve
$$x=\frac12 z^2,\qquad y=z+\sum_{k=1}^\infty h_{k}\frac {z^{2k+1}}{(2k+1)!!}.$$
\end{theorem}

In this statement, we treat $x(z)$ as a global function on the complex line $\bc$ with coordinate $z$ having a ramification of order $2$ at the origin, while $y(z)$ is just a formal power expansion near the point $z=0$.

\begin{corollary}
Any potential of topological recursion with the spectral curve whose $x$-function is given by $x(z)=\frac12z^2$ is a specialization of the $\kappa$-potential $\cf$ for suitable choices of the parameters $h_i$ (or $s_i$), and thus is equal to the Kontsevich-Witten potential $F^{KW}$ with a suitable shift of times $t_i$ for $i\ge2$.
\end{corollary}

%
%
\begin{proof}[Proof of Theorem~\ref{thTR}] Sum up all Virasoro equations of Corollary~\ref{Vir} multiplying the $m$th equation by $\frac{2 dz^2}{z^{2m+2}}$ where $z$ is a new variable. Then, denoting
$$\dtt(z)=\sum_{k=0}^\infty t_k\frac{z^{2k}dz}{(2k-1)!!},\quad
\eta(z)=y\,dx=\sum_{k=0}^\infty h_{k}\frac{z^{2k+2}dz}{(2k+1)!!}$$
and observing $\delta \cf=\sum_{k=0}^\infty\pd{\cf}{t_k}\frac{(2k+1)!!dz}{z^{2k+2}}$,
we collect all Virasoro equations into a single one that we call the \emph{master loop equation}
$$\hbar\Bigl(\delta^2\cf+\frac{dz^2}{4z^2}\Bigr)+(\delta\cf+\dtt(z)-\eta(z))^2=O(z^2),$$
in the Laurent expansion in $z$. The original Virasoro equations can be extracted from the above relation by selecting the coefficient of $dz^2z^{2i}$ for non-positive $i$. Equivalently, this equation can be written as
\begin{equation}\label{loop}
\delta\cf=\frac1{2\eta(z)}\left(\hbar\Bigl(\delta^2\cf+\frac{dz^2}{4z^2}\Bigr)+(\delta\cf)^2+2\,\dtt\,\cf+\dtt^2\right)+O(1)
\end{equation}

Now, select homogeneous terms of fixed degrees $(g,n)$ and apply $\delta_2\dots\delta_n$ to the obtained equation. Then the loop equation takes the following form. \emph{Redefine} in the unstable cases
    $$w_{0,1}(z_1)=0,\qquad w_{0,2}(z_1,z_2)=\delta_2\dtt(z_1)$$
and denote
$$\cw_{g,n}(z,z_N)=w_{g-1,n+1}(z,z,z_N)+
\sum_{\begin{smallmatrix}g_1+g_2=g\\I\sqcup J=N\end{smallmatrix}}w_{g_1,|I|+1}(z,z_I)w_{g_2,|J|+1}(z,z_J).$$
Here, we use the notation $N = \{2, 3, \ldots, n\}$ and $z_I = \{z_{i_1}, z_{i_2}, \ldots, z_{i_k}\}$ for $I = \{i_1, i_2, \ldots, i_k\}$. Then we obtain
\begin{equation}\label{recursion}
w_{g,n}(z,z_N)=\frac{\cw_{g,n}(z,z_N)}{2\eta(z)}+O(1)
\end{equation}
in the expansion in $z$ at $z=0$. Notice that $w_{g,n}(z,z_N)$ contains only negative powers of the coordinate $z$, therefore, the last equation determines it uniquely by induction: $w_{g,n}(z,z_N)$ as a meromorphic $1$-form in $z$ is the principal part of the pole at $z=0$ of the form on the right hand side of the equation. The obtained relation is nothing but an equivalent form of the conventional topological recursion relation. In order to observe it, we remark that the projection to the principal part of the pole can be written as the residue
$$w_{g,n}(z,z_N)=\mathop{\text{res}}_{\zeta=0}
\frac{\cw_{g,n}(\zeta,z_N)}{2\eta(\zeta)}\frac {\zeta\,dz}{z^2-\zeta^2}$$
Remark also that $w_{0,2}(z_1,z_2)$ is the odd part of the \emph{Bergman Kernel} $B(z_1,z_2)=\frac{dz_1dz_2}{(z_1-z_2)^2}$:
\begin{multline*}
w_{0,2}(z_1,z_2)=\delta_2\dtt(z_1)=\sum_{k=0}^\infty\frac{dz_2 (2 k+1)!!}{z_2^{2 k+2}} \frac{dz_1 z_1^{2 k}}{(2 k-1)!!} =\\
\frac{(z_1^2+z_2^2)dz_1dz_2}{(z_1^2-z_2^2)^2}=\frac12(B(z_1,z_2)-B(-z_1,z_2)),
\end{multline*}
and the form $\frac {\zeta\,dz}{z^2-\zeta^2}=\frac12\int_{-\zeta}^{\zeta}B(z,\cdot)$ is the odd part of its primitive.
\end{proof}

\subsection{Regularity of potentials}
We now introduce a parameter $\epsilon$ into various potentials, and prove that regularity in $\epsilon$ is equivalent to vanishing of some coefficients.  This is enough to prove Theorem~\ref{numeric} and reduces a weak version of Conjecture~\ref{conj} to a regularity property which is verified numerically in small cases.

\begin{proof}[Proof of Theorem~\ref{numeric}]
Let us make an additional rescaling in the series $\cf$ and in the corresponding correlator differentials $w_{g,n}$ by replacing the class $\ck(h)$ in their definition with $\sum_m \epsilon^{2g-2+n-m}\ck_m(h)$ where $\ck_m(h)$ is the degree~$m$ homogeneous component of $\ck(h)$. With this rescaling, the classes $\ck_m(h)$ with $m>2g-2+n$ are multiplied with negative powers of the parameter~$\epsilon$. Therefore, the fact that intersection numbers with these classes vanish is equivalent to the fact that the resulting rescaled polynomials $F_{g,n}$ and differentials $w_{g,n}$ are regular at $\epsilon=0$, that is, there exists limits $\lim\limits_{\epsilon\to 0}F_{g,n}$, or equivalently, there exist limits $\lim\limits_{\epsilon\to 0}w_{g,n}$. The rescaled correlator differentials also obey topological recursion with the following rescaled spectral curve
$$x=\frac12 z^2,\qquad y=\epsilon^{-1}\sum_{k=0}^\infty h_{k}\epsilon^{-k}\frac {z^{2k+1}}{(2k+1)!!}.$$
The rescaling $h_k\to h_k\epsilon^{-k}$ corresponds to the factor $\epsilon^{-m}$ of $\ck_m(h)$. And introducing an additional global factor $\epsilon^{-1}$ of $y$ results in the factor $\epsilon^{2g-2+n}$ for the whole class $\ck(h)$.

The class $\bk$ of Theorem~\ref{numeric} corresponds to the specialization $h_k=(-1)^k(2k+1)!!$ in $\ck(h)$. The $y$-function of the corresponding (rescaled) spectral curve is
\begin{equation}\label{yfuncK}
y=\epsilon^{-1}\sum_{k=0}^\infty (-1)^k(2k+1)!!\epsilon^{-k}\frac {z^{2k+1}}{(2k+1)!!}=\frac{z}{z^2+\epsilon}.
\end{equation}
Then we observe that this function possesses the property that
$$y^{-1}=z+\epsilon z^{-1}$$
is regular in $\epsilon$. This observation implies immediately the statement of Theorem~\ref{numeric} for the class~$\bk$ by induction since negative powers of $\epsilon$ do not appear in the recursion~\eqref{recursion}.
\end{proof}

\begin{remark} As was mentioned above, the topological recursion relation is equivalent to the loop equation which is, in turn, an equivalent form of Virasoro constraints. Therefore, the above arguments that use topological recursion for the correlator differentials can be translated to the arguments involving Virasoro constraints for the potential $\cf$. For the case of the $y$-function~\eqref{yfuncK} the loop equation~\eqref{loop} specializes to
$$
\delta\cf=(1+\epsilon z^{-2})
\frac1{2\,dz}\left(\hbar\Bigl(\delta^2\cf-\frac{dz^2}{4z^2}\Bigr)+(\delta\cf)^2+2\,\dtt\,\cf+\dtt^2\right)+O(1)
$$
Taking the coefficients of the fixed (negative) powers of~$z$, we obtain the Virasoro constraints in the following form
\begin{equation}\label{virK}
\bigl((2m+1)!!\pd{}{t_m}-\epsilon\;L_{m-1}-L_m\bigr)e^{\hbar^{-1}\cf}=0,\quad m=0,1,2,\dots.
\end{equation}
These equations determine the function uniquely. They are regular in~$\epsilon$, therefore, the solution is also regular in~$\epsilon$.
\end{remark}

\begin{remark}
Topological recursion is defined in \cite{EOrInv} for the spectral curve $x=\frac12 z^2$, $y=y(z)$
in the following form which is equivalent to the approach described in Section~\ref{sectcf}.  Define
\[
\omega_{0,2}(z_1, z_2) = \frac{dz_1dz_2}{(z_1-z_2)^2}.
\]
and define $\omega_{g,n}$ for $2g-2+n>0$ recursively via the following equation.
\begin{equation*}
\begin{split}
\omega_{g,n}&(z_1, z_S) \\
 & =   \mathop{\text{Res}}_{z=0} K(z_1, z) \bigg[ \omega_{g-1,n+1}(z, -z, z_S) + \hspace{-1mm}\sum_{\substack{h+h'=g \\ I \sqcup J = S}}^{\circ} \omega_{h,1 + |I|}(z, z_I) \, \omega_{h',1+|J|}(-z, z_{J}) \bigg].
\end{split}
\end{equation*}
Here, $S = \{2, 3, \ldots, n\}$ and $z_I = \{z_{i_1}, z_{i_2}, \ldots, z_{i_k}\}$ for $I = \{i_1, i_2, \ldots, i_k\} \subseteq S$. The symbol $\circ$ over the inner summation means that we exclude any term that involves $\omega_{0,1}$. The recursion kernel is defined by:
\[
K(z_1,z) = \frac{1}{2}\frac{\int_{-z}^{z} \omega_{0,2}(z_1,\cdot)}{[y(z) - y( -z))]d x(z)}.
\]
The proof of Theorem~\ref{numeric} restated in these terms uses the fact that $K(z_1,z)$ depends on $y^{-1}$ hence it is regular in $\epsilon$ which implies that each $\omega_{g,n}$ is also regular in $\epsilon$.
\end{remark}

\subsubsection{Intersections with the class $\bj$.}   \label{Jweak}
The class $\bj$ corresponds to the specialization $h_k=(-1)^kk!$ in $\ck(h)$. The $y$-function of the corresponding (rescaled) spectral curve is (it was first found in~\cite{EynInv})
$$y=\epsilon^{-1}\sum_{k=0}^\infty (-1)^kk!\epsilon^{-k}\frac {z^{2k+1}}{(2k+1)!!}=
\frac{2\text{\ arcsinh}(z/\sqrt{2\epsilon})}{\sqrt{z^2+2\epsilon}}.
$$

In this case, we have that
$$y^{-1}=\frac{\sqrt{z^2+2\epsilon}}{2\,\text{\ arcsinh}(z/\sqrt{2\epsilon})}=
\frac{\epsilon }{z}+\frac{z}{3}-\frac{z^3}{45 \epsilon }+\frac{z^5}{189 \epsilon ^2}-\frac{23 z^7}{14175 \epsilon ^3}+\dots
$$
is a Laurent series in $z$ whose coefficients are not regular in $\epsilon$ any more. It turns out, however, that due to some cancellation all negative powers of $\epsilon$ disappear in the course of the recursion and the differentials $w_{g,n}$ prove to be regular in $\epsilon$.
%
%
%
%
\begin{align*}
\omega_{0,3}&=\epsilon\prod_{i=1}^3\frac{dz_i}{z_i^2}\\
\omega_{0,4}&=\left(\epsilon+\epsilon^2\sum_{i=1}^4\frac{3}{z_i^2}\right)\prod_{i=1}^4\frac{dz_i}{z_i^2}\\
\omega_{1,1}&=\frac{1}{24}\left(1+\epsilon\frac{3}{z_1^2}\right)\frac{dz_1}{z_1^2}\\
\omega_{1,2}&=\frac{1}{8}\left(\epsilon\sum_{i=1}^2\frac{2}{z_i^2}+\epsilon^2\left(\sum_{i=1}^2\frac{5}{z_i^4}+\frac{3}{z_1^2z_2^2}\right)\right)\prod_{i=1}^2\frac{dz_i}{z_i^2}\\
\omega_{2,1}&=\frac{1}{1920}\left(-\frac{2}{z_1^2}+\epsilon\frac{130}{z_1^4}+\epsilon^{2}\frac{1015}{z_1^6}+\epsilon^3\frac{1575}{z_1^{8}}\right)\frac{dz_1}{z_1^2}\\
\end{align*}

We expect that the regularity in $\epsilon$ holds true for all $(g,n)$ with $2g-2+n>0$. Thus, the proof is reduced to a purely combinatorial problem which starts with an explicit spectral curve and asks for vanishing behaviour of some coefficients of its $\omega_{g,n}$.  

Let us turn to the (presumingly) top non-vanishing term.
Consider the function
$$F^J(\hbar,t_0,t_1,...)=\sum_{g,n,\vec{k}}\frac{\hbar^{g}}{n!}\int_{\overline{\modm}_{g,n}}J_{2g-2+n}\cdot\prod_{j=1}^n\psi_j^{k_j}t_{k_j}.
$$
Corresponding to the (conjectural) limit $\epsilon\to0$ for the above specialization of the potential $\cf$. The numerical experiments show that the $\epsilon=0$ limit of $w_{g,n}$ can be nonvanishing for $n=1$ only. This means that all terms of $F^J$ of homogeneous degree $>1$ in $t$-variables are vanishing,
\begin{multline*}
F^J(\hbar,t_0,t_1,...)=\frac{t_0}{24}\hbar-\frac{t_1}{2880}\hbar^2+\frac{t_2}{120960}\hbar^3-\frac{t_3}{3225600}\hbar^4+\\
\frac{t_4}{63866880}\hbar^5-\frac{691}{697426329600}t_5\hbar^6+\frac{1}{13284311040}t_6\hbar^7+\dots
\end{multline*}

Hence it seems that $F^J$ is linear in $t_i$ and given by the unproven formula
\begin{equation}  \label{Jpot}
F^J(\hbar,t_0,t_1,...)=\sum_{g>0}\frac{(g-1)!B_{2g}}{(2g)!2^g}\cdot\hbar^{g}t_{g-1}
\end{equation}
where $B_{2g}$ is a Bernoulli number.
It is proven in \cite[Theorem 1]{FPaLog} that
$$\int_{\overline{\modm}_{g,1}}\lambda_{g-1}\lambda_g\psi^{g-1}=\int_{\overline{\modm}_{g}}\lambda_{g-1}\lambda_g\kappa_{g-2}=\frac{(-1)^{g-1}}{2^{2g-1}(2g-1)!!}\frac{B_{2g}}{2g}
$$
and from \eqref{Jpot}, which is only calculated numerically for small $g$,
$$\int_{\overline{\modm}_{g,1}}J_{g,1}\psi^{g-1}=\frac{1}{2^{2g-1}(2g-1)!!}\frac{B_{2g}}{2g}
$$
so they agree up to $(-1)^{g-1}$.
These properties give evidence towards Conjecture~\ref{conjhodge}, that $J_{g,1}=(-1)^{g-1}\lambda_{g-1}\lambda_g$.  This is discussed further in Section~\ref{hodge}.

\subsubsection{Intersections with $\kappa$ classes}

We now discuss verification of trivial intersection pairings of the classes $K_m$ and $J_m$ for $m>2g-2+n$ with arbitrary monomials in $\psi$ classes and $\kappa$ classes. Let us denote by $h^*$ the specialization of Theorem~\ref{numeric} of $h$ parameters, that is, either $h_k^*=(-1)^k(2k-1)!!$ or $h_k^*=(-1)^kk!$. Consider the potential enumerating intersection numbers of the class $\ck(h^*)$ with both $\psi$ classes and $\kappa$ monomials. It involves additional parameters $h_1,h_2,\dots$ and is obtained by replacement of the class $\ck$ in the definitions of $\cf$ and $w_{g,n}$ with the product $\ck(h^*)\ck(h)$. The product $\ck(h^*)\ck(h)$ of two exponents corresponds to the shift $s_k\to s_k+s_k^*$ of $s$-coordinates in the class $\ck$, or, in terms of $h$-coordinates, it corresponds to the change $h_k\to\sum_{i+j=k}h_i^*h_j$. Taking also into account the $\epsilon$-rescaling separating intersection numbers with the classes $\ck_m(h^*)$ with $m\le 2g-2+n$ and those with $m>2g-2+n$, we arrive at the following assertion.
\begin{proposition}
The correlator differentials $w_{g,n}$ enumerating intersection numbers with the classes $\bk$ and $\bj$ of Theorem~\ref{numeric} with arbitrary monomials in $\psi$ classes and $\kappa$ classes satisfy relations of topological recursion with the spectral curve
\begin{equation}  \label{weakspecurve}
x=\frac12 z^2,\qquad y=\sum_{k=0}^\infty\frac {z^{2k+1}}{(2k+1)!!}\sum_{i+j=k}\epsilon^{-i-1}h_i^*h_j
\end{equation}
where the sum is over $i,j\in\bn=\{0,1,2,...\}$ and we set $h_0^*=1=h_0$.  Choose
$h_k^*=(-1)^k(2k-1)!!$ for the case $\bk$ and  $h_k^*=(-1)^kk!$ for the case $\bj$, while $h_1,h_2,\dots$ are formal parameters.
\end{proposition}
Starting the recursion we find the differentials $w_{g,n}$ and thus the components $F_{g,n}$ of the partition function for small $g$ and $n$. For the class $\bk$ they are as follows.  (We prefer to demonstrate $F_{g,n}$'s rather than $w_{g,n}$'s since they are more compact)
\begin{align*}
F_{0,3}&=\frac{t_0^3 \epsilon }{6},\\
F_{1,1}&=\epsilon  \left(\frac{t_1}{24}-\frac{h_1 t_0}{24}\right)+\frac{t_0}{8},\\
F_{0,4}&=\epsilon ^2 \left(\frac{1}{6} t_0^3 t_1-\frac{1}{24} h_1 t_0^4\right)+\frac{t_0^4 \epsilon }{8},\\
F_{1,2}&=\epsilon ^2 \left(\frac{1}{48} \left(2 h_1^2-h_2\right) t_0^2+t_0 \left(\frac{t_2}{24}-\frac{h_1 t_1}{12}\right)+\frac{t_1^2}{48}\right)+\epsilon  \left(\frac{t_0 t_1}{4}-\frac{3}{16} h_1 t_0^2\right)+\frac{t_0^2}{16},\\
F_{2,1}&=\epsilon ^3 \left(..
\right)+
\epsilon ^2 \left(..
\right)+\epsilon  \left(\frac{\left(627 h_1^2-203 h_2\right) t_0}{1920}-\frac{407 h_1 t_1}{1920}+\frac{107 t_2}{1920}\right)\hspace{-.5mm}-\hspace{-.5mm}\frac{9 h_1 t_0}{128}+\frac{3 t_1}{128}.
\end{align*}

 For the class $\bj$ these components are
\begin{align*}
F_{0,3}&=\frac{t_0^3 \epsilon }{6},\\
F_{1,1}&=\epsilon  \left(\frac{t_1}{24}-\frac{h_1 t_0}{24}\right)+\frac{t_0}{24},\\
F_{0,4}&=\epsilon ^2 \left(\frac{1}{6} t_0^3 t_1-\frac{1}{24} h_1 t_0^4\right)+\frac{t_0^4 \epsilon }{24},\\
F_{1,2}&=\epsilon ^2 \left(\frac{1}{48} \left(2 h_1^2-h_2\right) t_0^2+t_0 \left(\frac{t_2}{24}-\frac{h_1 t_1}{12}\right)+\frac{t_1^2}{48}\right)+\epsilon  \left(\frac{t_0 t_1}{12}-\frac{1}{16} h_1 t_0^2\right),\\
F_{2,1}&=\epsilon ^3 \left(..
\right)+
\epsilon ^2 \left(..
\right)+\epsilon  \left(\frac{\left(157 h_1^2-50 h_2\right) t_0}{5760}-\frac{17 h_1 t_1}{960}+\frac{13 t_2}{2880}\right)+\frac{h_1 t_0}{2880}-\frac{t_1}{2880}.
\end{align*}

In both cases these polynomials contain only non-negative powers of $\epsilon$. This does not follow automatically from the recursion since in both cases the Laurent expansion of $y^{-1}$ in $z$ contains negative powers of $\epsilon$. However, we expect that these polynomials are regular in $\epsilon$ for all $g$ and $n$ which is equivalent to the fact that the intersection with classes $K_m$ and $J_m$ for $m>2g-2+n$ is trivial for arbitrary tautological classes.  Moreover, we observe that in the $J$ case the limit for $\epsilon\to0$ is nonvanishing in the case $n=1$ only, as expected.

We say that a class in the tautological ring $RH^*(\overline{\mathcal{M}}_{g,n})\subset H^{2*}(\overline{\mathcal{M}}_{g,n},\bq)$ vanishes {\em weakly} if  it vanishes under the intersection pairing with all other classes in $RH^*(\overline{\mathcal{M}}_{g,n})$.  In particular, in those cases where the tautological ring $RH^*(\overline{\mathcal{M}}_{g,n})$ is perfect Conjecture~\ref{conj} follows from its weak version.
\begin{corollary}  \label{regweak}
If the correlators of the spectral curve \eqref{weakspecurve} are regular in $\epsilon$ then Conjecture~\ref{conj} holds weakly, i.e. the intersection pairing of $K_m$ with all classes in $RH^*(\overline{\mathcal{M}}_{g,n})$ vanishes for $m>2g-2+n$.
\end{corollary}
\begin{proof}
Let $G_{g,n}$ be the finite set of stable graphs of type $(g,n)$.  Each $\Gamma\in G_{g,n}$ defines a stratum $\phi_\Gamma:\overline{\modm}_\Gamma\to\overline{\modm}_{g,n}$ and any tautological class $\omega\in RH^m(\overline{\mathcal{M}}_{g,n})$ is given by a sum over stable graphs
\[\omega=\sum_{\Gamma\in G_{g,n}}(\phi_\Gamma)_*p_\Gamma
\]
where $p_\Gamma=\bigotimes_{v\in\Gamma}p_v$ is a product of polynomials $p_v$ in $\psi$ classes and $\kappa$ classes defined at each vertex $v$ of $\Gamma$.  Now
\[ \int_{\overline{\modm}_{g,n}}\bk\cdot\omega=\sum_{\Gamma\in G_{g,n}}\int_{\overline{\modm}_{g,n}}\bk\cdot(\phi_\Gamma)_*p_\Gamma 
=\sum_{\Gamma\in G_{g,n}}\int_{\overline{\modm}_\Gamma}\bk_\Gamma\cdot p_\Gamma
\]
where the second equality uses $\phi_\Gamma^*\bk=\bk_\Gamma$ which consists of a factor of $\bk$ at each vertex of $\Gamma$.
Let $E(\Gamma)$ and $V(\Gamma)$ be the number of interior edges, respectively vertices, of $\Gamma$.   When $V(\Gamma)=1$ and $E(\Gamma)=0$, $\omega$ is a polynomial in $\psi$ classes and $\kappa$ classes, and the vanishing of $\bk\cdot\omega$ follows from regularity in $\epsilon$ of the correlators of the spectral curve \eqref{weakspecurve}.

More generally, if $V(\Gamma)=1$ and $E(\Gamma)>0$, then
\[ g-1>\deg\omega=E(\Gamma)+\deg p_v=g-g_v+\deg p_v\]
hence $\deg p_v<g_v-1$, and
\[\int_{\overline{\modm}_{g,n}}\bk\cdot\omega=\int_{\overline{\modm}_{g_v,n_v}}\bk\cdot p_v=0\]
where $(g_v,n_v)=(g-E(\Gamma),n+2E(\Gamma))$ and the second equality follows from the initial case of $V(\Gamma)=1$ and $E(\Gamma)=0$.

In full generality, we have
\[ g-1>\deg\omega=E(\Gamma)+\sum_v\deg p_v
\]
hence there exists a vertex $v$ for which $p_v<g_v-1$ since otherwise
\[ g-1>\deg\omega\geq E(\Gamma)+\sum_v(g_v-1)=E(\Gamma)-V(\Gamma)+\sum_vg_v=b_1(\Gamma)+\sum_vg_v=g\]
which is a contradiction.  On that vertex, $\int_{\overline{\modm}_{g(v),n(v)}}\bk\cdot p_v=0$ by the $V(\Gamma)=1$ cases proven above.  Hence $\bk\cdot\omega=0$ as required.
\end{proof}

\section{Relation to the Br\'ezin-Gross-Witten tau function}  \label{sec:relbgw}

\subsection{KdV tau functions}   \label{sec:kdv}

A tau function $Z(t_0,t_1,...)$ of the KdV hierarchy (equivalently the KP hierarchy in odd times $p_{2k+1}=t_k/(2k+1)!!$) gives rise to a solution $U=\hbar\frac{\partial^2}{\partial t_0^2}\log Z$ of the KdV hierarchy
\begin{equation}\label{kdv}
U_{t_1}=UU_{t_0}+\frac{\hbar}{12}U_{t_0t_0t_0},\quad U(t_0,0,0,...)=f(t_0).
\end{equation}
The first equation in the hierarchy is the KdV equation \eqref{kdv}, and later equations $U_{t_k}=P_k(U,U_{t_0},U_{t_0t_0},...)$ for $k>1$ determine $U$ uniquely from $U(t_0,0,0,...)$.  See \cite{MJDSol} for the full definition.

The Br\'ezin-Gross-Witten solution $U^{\text{BGW}}=\hbar\frac{\partial^2}{\partial t_0^2}\log Z^{\text{BGW}}$ of the KdV hierarchy arises out of a unitary matrix model studied in \cite{BGrExt,GWiPos}.  It is defined by the initial condition
\begin{equation*}
\hbar\frac{\partial^2}{\partial t_0^2}\log Z^{\text{BGW}}(t_0,0,0,...)=U^{\text{BGW}}(t_0,0,0,...)=\frac{\hbar}{8(1-t_0)^2}.
\end{equation*}
It arises out of topological recursion applied to the spectral curve $(x,y)=(\frac12z^2,1/z)$ \cite{DNoTop} and more generally to any spectral curve with points locally of this form near a zero of $dx$, \cite{CNoTop}.

The low genus $g$ terms (= coefficient of $\hbar^{g-1}$) of $\log Z^{\text{BGW}}$ are
\begin{align*}
\log Z^{\text{BGW}}&=-\frac{1}{8}\log(1-t_0)+\hbar\frac{3}{128}\frac{t_1}{(1-t_0)^3}+\hbar^2\frac{15}{1024}\frac{t_2}{(1-t_0)^5}\\
&\hspace{4.5cm}+\hbar^2\frac{63}{1024}\frac{t_1^2}{(1-t_0)^6}+O(\hbar^3) \\
&=(\frac{1}{8}t_0+\frac{1}{16}t_0^2+\frac{1}{24}t_0^3+...)+\hbar(\frac{3}{128}t_1+\frac{9}{128}t_0t_1+...)\\
&\hspace{4cm}+\hbar^2(\frac{15}{1024}t_2+\frac{63}{1024}t_1^2+...)+...
\end{align*}

The tau function $Z^{\text{BGW}}(\hbar,t_0,t_1,...)$ shares many properties of the famous Kont\-sevich-Witten tau function $Z^{\text{KW}}(\hbar,t_0,t_1,...)$ introduced in \cite{WitTwo}.  The Kontsevich-Witten tau function $Z^{\text{KW}}$ is defined by the initial condition $U^{\text{KW}}(t_0,0,0,...)=t_0$ for $U^{\text{KW}}=\hbar\frac{\partial^2}{\partial t_0^2}\log Z^{\text{KW}}$.  It arises out of topological recursion applied to the curve $(x,y)=(\frac12z^2,z)$ \cite{EOrTop}.The low genus terms of $\log Z^{\text{KW}}$ are
$$\log Z^{\text{KW}}(\hbar,t_0,t_1,...)=\hbar^{-1}(\frac{t_0^3}{3!}+\frac{t_0^3t_1}{3!}+\frac{t_0^4t_2}{4!}+...)+\frac{t_1}{24}+...
$$
\begin{theorem}[Witten-Kontsevich 1992 \cite{KonInt,WitTwo}]   \label{KW}
$$Z^{\text{KW}}(\hbar,t_0,t_1,...)=\exp\sum_{g,n}\hbar^{g-1}\frac{1}{n!}\sum_{\vec{k}\in\bn^n}\int_{\overline{\modm}_{g,n}}\prod_{i=1}^n\psi_i^{k_i}t_{k_i}
$$
is a tau function of the KdV hierarchy.
\end{theorem}

\subsection{BGW tau function from the limit of intersection numbers}

By the proof of Theorem~\ref{numeric} the $\epsilon$-rescaled potential for intersection numbers of the class $\bk$ with $\psi$ monomials, with the constant term removed, admits a limit for $\epsilon\to 0$. The limit is a similar potential enumerating intersection numbers of $\psi$ monomials with the component $K_{2g-2+n}$ of $\bk$. Denote this limit by $F^K$,
$$F^K(\hbar,t_0,t_1,...)=\sum_{g,n,\vec{k}}\frac{\hbar^{g}}{n!}\int_{\overline{\modm}_{g,n}}K_{2g-2+n}\cdot\prod_{j=1}^n\psi_j^{k_j}t_{k_j}.
$$

\begin{theorem} \label{limBGW}
$e^{\hbar^{-1}F^K}$ is the BGW tau function.
\end{theorem}

\begin{proof}
We provide two proofs of this fact. First, by (the proof of) Theorem~\ref{numeric}, the potential for the intersection numbers with the whole class $\bk$ is govern by the topological recursion with the spectral curve $(x,y)=(\frac12z^2,\frac{z}{\epsilon+z^2})$ which is equivalent to the Virasoro constraints~\eqref{virK}. This implies the existence of the limit $\epsilon\to 0$ for the potential, and also the fact that the limit satisfies Virasoro equations of the form
$$
\bigl((2m+1)!!\pd{}{t_m}-L_m\bigr)e^{\hbar^{-1}F^K}=0,\quad m=0,1,2,\dots.
$$
These are exactly Virasoro constraints for the BGW potential, and these equations determine the function uniquely, see e.g.~\cite{GNeUni}. As a byproduct of this proof, we recover the spectral curve $(x,y)=(\frac12z^2,\frac{1}{z})$ for the BGW potential $F^K$ found in~\cite{EOrTop}.

For the second proof, we observe that the potential $\cf$ of Sect.~\ref{sectcf} being the Kontsevich-Witten potential with shifted times, is a solution to the KdV hierarchy for any choice of parameters $h_i$. It follows that the potential $F^K$ as a limit of a family of solutions is also a solution of the KdV hierarchy.

It remains to check the initial conditions for $F^K|_{t_1=t_2=\dots=0}$. This part of the potential enumerates intersection numbers of $K_{2g-2+n}$ with no insertion of $\psi$ classes, that is, those of the form $\int_{\overline{\modm}_{g,n}}K_{2g-2+n}$. By dimensional reason, these numbers are nonzero in the case $g=1$ only, and by Theorem~\ref{KWshift} we have
$$F^K(t_0,0,0,\dots)=\hbar\sum_{n=0}^\infty\int_{\overline{\modm}_{1,n}}K_{n}\frac{t_0^n}{n!}=
\hbar F_1^{\text{KW}}(t_0,0,3!!,-5!!,7!!,\dots).$$
Fortunately, the $g=1$ component $F_1^{\text{KW}}$ of the Kontsevich-Witten potential is known. By the Itzykson-Zuber anzatz~\cite{IZuCom} (a sequence of the string and dilaton equations), it is given by
$$F_1^{\text{KW}}=-\frac1{24}\log\Bigl(1-\sum_{k=0}^\infty t_{k+1}\frac{Y^k}{k!}\Bigr)$$
where $Y=Y(t_0,t_1,\dots)$ is the solution of the implicit equation
$$Y=\sum_{k=0}^\infty  t_{k}\frac{Y^k}{k!}.$$
Denote $y(t_0)=Y(t_0,0,3!!,-5!!,7!!,\dots)$. Then the implicit equation specializes to
$$y=1+\sum_{k=2}^\infty(-1)^k(2k-1)!!\frac{y^k}{k!}=t_0+y+\frac{1}{\sqrt{1+2 y}}-1$$
which gives $\frac{1}{\sqrt{2 y+1}}=1-t_0$. Then, $1-\sum_{k=0}^\infty t_{k+1}\frac{Y^k}{k!}$ specializes to
$$1-\sum_{k=2}^\infty(-1)^k(2k-1)!!\frac{y^{k-1}}{(k-1)!}=\frac{1}{(1+2 y)^{3/2}}=(1-t_0)^3,$$
and we get finally
$$F^K(t_0,0,0,\dots)=-\frac{\hbar}{24}\log(1-t_0)^3=-\frac{\hbar}{8}\log(1-t_0),$$
$$\pd{^2F^K(t_0,0,0,\dots)}{t_0^2}=\frac{\hbar}{8(1-t_0)^2}.$$
This coincides with the initial condition for the BGW potential.
\end{proof}

In Section~\ref{spincohft} we relate Theorem~\ref{limBGW} to a conjectural relationship between a family of classes $\Theta_{g,n}\in H^{2(2g-2+n)}(\overline{\modm}_{g,n},\bq)$ and the BGW potential.

\section{Cohomological field theories}  \label{sec:cohft}

A {\em cohomological field theory} is a pair $(V,\eta)$ composed of a finite-dimensional complex vector space $V$ equipped with a metric $\eta$ and a sequence of $S_n$-equivariant maps.
\[ \Omega_{g,n}:V^{\otimes n}\to H^*(\overline{\modm}_{g,n},\bq)\]
that satisfy compatibility conditions from inclusion of strata:
$$\phi_{\text{irr}}:\overline{\modm}_{g-1,n+2}\to\overline{\modm}_{g,n},\quad \phi_{h,I}:\overline{\modm}_{h,|I|+1}\times\overline{\modm}_{g-h,|J|+1}\to\overline{\modm}_{g,n},\ \  I\sqcup J=\{1,...,n\}$$
given by
\begin{align}
\phi_{\text{irr}}^*\Omega_{g,n}(v_1\otimes...\otimes v_n)&=\Omega_{g-1,n+2}(v_1\otimes...\otimes v_n\otimes\Delta)
  \label{glue1}
\\
\phi_{h,I}^*\Omega_{g,n}(v_1\otimes...\otimes v_n)&=\pi_1^*\Omega_{h,|I|+1}\otimes \pi_2^*\Omega_{g-h,|J|+1}\big(\bigotimes_{i\in I}v_i\otimes\Delta\otimes\bigotimes_{j\in J}v_j\big) \label{glue2}
\end{align}
where $\Delta\in V\otimes V$ is dual to the metric $\eta\in V^*\otimes V^*$, and $\pi_i$, $i=1,2$ is projection onto the $i$th factor of $\overline{\modm}_{h,|I|+1}\times\overline{\modm}_{g-h,|J|+1}$.  When $n=0$, $\Omega_g:=\Omega_{g,0}\in H^*(\overline{\modm}_{g},\bq)$.   
The CohFT has a {\em unit} if there exists a vector $\un\in V$ which satisfies
\begin{equation}  \label{unit}
\Omega_{0,3}(\un\otimes v_1\otimes v_2)=\eta(v_1,v_2).
\end{equation}
We also consider the case when $\Omega_{g,n}$ satisfies \eqref{glue1} and \eqref{glue2} but not \eqref{unit} and refer to this as a CohFT without unit.

The CohFT has {\em flat unit} if $\un\in H$ is compatible with the forgetful map $\pi:\overline{\modm}_{g,n+1}\to\overline{\modm}_{g,n}$ by
\begin{equation}   \label{cohforget}
\Omega_{g,n+1}(\un\otimes v_1\otimes...\otimes v_n)=\pi^*\Omega_{g,n}(v_1\otimes...\otimes v_n)
\end{equation}
for $2g-2+n>0$.

Given a CohFT $\Omega=\{\Omega_{g,n}\}$ and a basis $\{e_1,...,e_N\}$ of $V$, the partition function of $\Omega$ is defined by:
\begin{equation*}   
Z_{\Omega}(\hbar,\{t^{\alpha}_k\})=\exp\sum_{g,n,\vec{k}}\frac{\hbar^{g-1}}{n!}\int_{\overline{\modm}_{g,n}}\Omega_{g,n}(e_{\alpha_1}\otimes...\otimes e_{\alpha_n})\cdot\prod_{j=1}^n\psi_j^{k_j}\prod t^{\alpha_j}_{k_j}
\end{equation*}
for $\alpha_i\in\{1,...,N\}$ and $k_j\in\bn$.

In the case of one-dimensional CohFTs, i.e. $\dim V=1$, with unit $\un\in V$ that satisfies $\eta(\un,\un)=1$ (which is related to the more general $\eta(\un,\un)\in\bc$ by a simple rescaling), we identify $\Omega_{g,n}$ with its image $\Omega_{g,n}(\un^{\otimes n})$, and write $\Omega_{g,n}\in H^*(\overline{\modm}_{g,n},\bq)$.

Given two CohFTs $\Omega^i$ defined on $(V_i,\eta_i)$, $i=1,2$, they define a CohFT $\Omega^1\otimes\Omega^2$ on $(V_1\otimes V_2,\eta_1\otimes\eta_2)$ via the cup products $\Omega^1_{g,n}(u_1\otimes...\otimes u_n)\cup\Omega^2_{g,n}(v_1\otimes...\otimes v_n)$, and known as the tensor product.  The trivial CohFT $\un$ defined by $\un_{g,n}=1\in H^0(\overline{\modm}_{g,n},\bq)$ acts as a unit with respect to tensor product.  An {\em invertible} CohFT is any one-dimensional CohFT $\Omega^1$ such that there exists $\Omega^2$, its inverse, with $\Omega^1\otimes\Omega^2=\un$.   Invertible CohFTs were classified by Manin-Zograf \cite{MZoInv} and Teleman \cite{TelStr}.
From the properties $\phi_{\text{irr}}^*\kappa_m=\kappa_m$ and $\phi_{h,I}^*\kappa_m=\pi_1^*\kappa_m+\pi_2^*\kappa_m$, proven in \cite{ACoCom}, and $\phi_{\text{irr}}^*\mu_m=\mu_m$ and $\phi_{h,I}^*\mu_m=\pi_1^*\mu_m+\pi_2^*\mu_m$, for $\mu_k=ch_k(\be)$, which uses the decomposition of the restriction of the Hodge bundle to the boundary, Manin and Zograf \cite{MZoInv} proved that $\exp{\sum_{i\geq 0}(s_i\kappa_i+t_i\mu_{2i+1})}$ defines an invertible CohFT.  The converse was proven by Teleman in \cite{TelStr} giving the following theorem.
\begin{theorem}[\cite{MZoInv,TelStr}]
For any $ \{s_i,t_i\}_{i\in\bn}\subset\bc$
\[\Omega_{g,n}=\exp{\sum_{i\geq 0}(s_i\kappa_i+t_i\mu_{2i+1})}\]
defines an invertible CohFT and any invertible CohFT is of this form.
\end{theorem}
In particular, $\Omega_{g,n}^\bk=\bk$ and $\Omega_{g,n}^\bj=\bj$ define invertible CohFTs.

\subsection{Push-forward and pull-back properties}   \label{sec:prop}

The polynomials $K_m(\kappa_1,...,\kappa_m)$ and $J_m(\kappa_1,...,\kappa_m)$ satisfy push-forward and pull-back relations with respect to the forgetful map $\pi:\overline{\modm}_{g,n+1}\to\overline{\modm}_{g,n}$.
\begin{equation}  \label{pushf}
\pi_*J_{m+1}=(2g-2+n-m)J_{m},\quad
\pi_*K_{m+1}=(6g-6+3n-2m)K_{m}.
\end{equation}
\begin{equation}  \label{pullb}
\pi^*J_m=\sum_{i=0}^m(-1)^ii!\psi_{n+1}^iJ_{m-i},\qquad\pi^*K_m=\sum_{i=0}^m(-1)^i(2i+1)!!\psi_{n+1}^iK_{m-i}.
\end{equation}
The push-forward and pull-back relations are a consequence of a more general result.
For any $v=(a,b)\in\bc^2$, define
$$\bl^v=L^v_0+L^v_1t+L^v_2t^2+...=\exp(\sum\ell^v_j\kappa_jt^j)$$
for $\ell^v_j$ defined by \eqref{de} and \eqref{discon}.
\begin{deflem}  \label{reciprocal}
For any $v=(a,b)\in\bc^2$, define the sequence $\{\ell^v_j\mid j>0\}$ in either of the following two equivalent ways:
\begin{equation}  \label{de}
\exp\left(\sum\ell_j^vt^j\right)=1+at-bt^2\frac{d}{dt}\sum_{j>0}\ell_j^vt^j
\end{equation}
\begin{equation}  \label{discon}
\exp\left(-\sum\ell_j^vt^j\right)=\sum_{k=0}^\infty (-1)^kt^k\prod_{m=0}^{k-1}(a+bm).
\end{equation}
\end{deflem}
\begin{proof}
Define $f(t)$ by the right hand side of~\eqref{discon}. Then, shifting the summation index by~$1$, we get
\begin{multline*}
f(t)=1+\sum_{k=0}^\infty (-1)^{k+1}t^{k+1}\prod_{m=0}^{k}(a+mb)=\\
1-t\sum_{k=0}^\infty (a+b k)(-1)^{k}t^{k}\prod_{m=0}^{k-1}(a+mb)=
1-t(a+b\,t\frac{d}{dt})f(t).
\end{multline*}
The obtained differential equation is equivalent to~\eqref{de}.
\end{proof}
The classes $L^v_k$ for any $v=(a,b)\in\bc^2$ are proven below to satisfy push-forward relations using \eqref{de} and pull-back relations using \eqref{discon}.   These relations specialise to \eqref{pushf} and \eqref{pullb} since $\bl^{(1,1)}=\bj$ and $\bl^{(3,2)}=\bk$, i.e.
if we substitute $(a,b)=(1,1)$ or $(3,2)$ into \eqref{discon}, the right hand side becomes
$$\sum_{k=0}^\infty (-1)^kt^k\prod_{m=0}^{k-1}(1+m)=\sum_{k=0}^\infty(-1)^kk!t^k$$
and
$$\sum_{k=0}^\infty (-1)^kt^k\prod_{m=0}^{k-1}(3+2m)=\sum_{k=0}^\infty(-1)^k(2k+1)!!t^k
$$
and the relations \eqref{pushf} and \eqref{pullb} are immediate corollaries of Lemmas~\ref{pushfgen} and \ref{pullbgen}.  We have abused notation and identified $\bk(t)=K_0+K_1t+K_2t^2+...$ with $\bk=\bk(1)$ since $t$ simply keeps track of degree, and similarly with $\bj$.

\begin{lemma} \label{pushfgen}
For any $(a,b)\in\bc$
$$\pi_*L^v_{m+1}=a(2g-2+n)L^v_{m}-bmL^v_{m}.$$
\end{lemma}
\begin{proof}
Each $L^v_m$ is a degree $m$ polynomial in $\{\kappa_j\}$.  We can calculate its push-forward as follows:
\begin{align*}
\pi_*\bl^v
=&\pi_*\exp\sum \ell^v_i\kappa_it^i=\pi_*\exp\sum \ell^v_i(\pi^*\kappa_i+\psi^i_{n+1})t^i\\
=&\pi_*\left(\exp(\sum \ell^v_i\pi^*\kappa_it^i)\exp(\sum \ell^v_i\psi^i_{n+1}t^i)\right)\\
=&\pi_*\left(\pi^*(\exp(\sum \ell^v_i\kappa_it^i))\exp(\sum \ell^v_i\psi^i_{n+1}t^i)\right)\\
=&\bl^v\cdot\pi_*\exp\sum \ell^v_i\psi^i_{n+1}t^i\\
=&\bl^v\cdot\pi_*\left(1+a\psi_{n+1}t-b\sum_{j>0} j\ell^v_j\psi^{j+1}_{n+1}t^{j+1}\right)\\
=&\bl^v\cdot\left(a(2g-2+n)t-b\sum_{j>0} j\ell^v_j\kappa_jt^{j+1}\right)\\
=&\bl^v\cdot a(2g-2+n)t-b\sum kL^v_kt^{k+1}.
\end{align*}
The third last equality uses the differential equation \eqref{de} with $t\mapsto\psi_{n+1}t$.
The last equality is obtained via differentiation with respect to $t$ of $\bl^v=\sum L^v_kt^k=\exp(\sum\ell^v_j\kappa_jt^j)$ to get
$$ \sum kL^v_kt^{k-1}=\bl^v\cdot\sum_{j>0}j\ell^v_j\kappa_jt^{j-1}.
$$
The coefficient of $t^{m+1}$ in the expression for $\pi_*\bl^v$ yields $\pi_*L^v_{m+1}=a(2g-2+n)L^v_{m}-bmL^v_{m}$ as required.
\end{proof}
\begin{lemma} \label{pullbgen}
For any $(a,b)\in\bc$
$$\pi^*L^v_k=\sum_{i=1}^k(-1)^i\psi_{n+1}^iL^v_{k-i}\prod_{m=0}^{i-1}(a+mb).$$
\end{lemma}
\begin{proof}
\begin{align*}
\pi^*\bl^v
&=\pi^*\exp(\sum\ell^v_i\kappa_it^i)=\exp(\sum\ell^v_i\pi^*\kappa_it^i)=\exp(\sum \ell^v_i(\kappa_i-\psi_{n+1}^i)t^i\\
&=\bl^v\exp\left(-\sum\ell^v_i\psi_{n+1}^it^i)\right)\\
&=\bl^v\sum_{i=1}^\infty (-1)^i\psi_{n+1}^it^i\prod_{m=0}^{i-1}(a+mb)
\end{align*}
where the last equality uses \eqref{discon}.  The coefficient of $t^k$ yields the lemma.
\end{proof}
\begin{remark}
There is a combinatorial proof of Lemma~\ref{reciprocal} in the case of $v=(1,1)$ which considers connected and disconnected  permutations of $(1,...,n)$.
\end{remark}

\subsection{CohFTs without unit}
A one-dimensional CohFT without unit is a collection of classes $\Omega_{g,n}\in H^*(\overline{\modm}_{g,n},\bq)$ that satisfies \eqref{glue1} and \eqref{glue2}.  It may be that $\Omega_{0,3}=0$ as in the following examples.
\begin{proposition}  \label{equivconj}
The sequences $\Omega_{g,n}=K_{2g-2+n}$ and $\Omega'_{g,n}=J_{2g-2+n}$ are CohFTs without unit and $\Omega'_{g,n}=0$ for $n>1$ iff Conjecture~\ref{conj} holds.
\end{proposition}
\begin{proof}
We will treat the case of $\Omega_{g,n}=K_{2g-2+n}$ first.  One direction is straightforward.  If
\begin{equation}  \label{vaneq}
K_m(\kappa_1,...,\kappa_m)=0\quad\text{for}\quad m>2g-2+n,\ n>0
\end{equation}
then for $\phi:\overline{\modm}_{g_1,n_1+1}\times\overline{\modm}_{g_2,n_2+1}\to\overline{\modm}_{g,n}$ where $g_1+g_2=g$, $n_1+n_2=n$ and $p_i$, $i=1,2$ is projection onto the $i$th factor
\begin{equation}  \label{conjres}
\phi^*K_{2g-2+n}=\sum_{i+j=2g-2+n}p_1^*K_i\cdot p_2^*K_j=p_1^*K_{2g_1-2+n_1+1}\cdot p_2^*K_{2g_2-2+n_2+1}.
\end{equation}
The first equality uses $\phi^*\exp(\sum s_i\kappa_i)=p_1^*\exp(\sum s_i\kappa_i)\cdot p_2^* \exp(\sum s_i\kappa_i)$.
For the second equality, $i+j=2g-2+n$ implies that for all but one summand in \eqref{conjres} either $i> 2g_1-2+n_1+1$, hence $K_i=0$ by \eqref{vaneq}, or $j> 2g_2-2+n_2+1$ in which case $K_j=0$ again by \eqref{vaneq}.  Also $\phi_{\text{irr}}^*\exp(\sum s_i\kappa_i)=\exp(\sum s_i\kappa_i)$ implies that $\phi_{\text{irr}}^*K_{2g-2+n}=K_{2(g-1)-2+n+2}$ and together with \eqref{conjres} we see that $\Omega_{g,n}=K_{2g-2+n}$ is a CohFT without unit.

For the other direction, suppose that $\Omega_{g,n}=K_{2g-2+n}$ is a CohFT without unit.  For  $k>0$ and $n>0$,  consider the following stable graph $\Gamma\in G_{g,n+k}$:
\begin{center}
\begin{tikzpicture}
\node(1) at (-1,0)        {$\Gamma=$};
\node(3) at (1,0)[shape=circle,draw,scale=.7]        {$g$};
\node(4) at (2,0)[shape=circle,draw,scale=.7]        {0};
\node(5) at (4,0)[shape=circle,draw,scale=.7]        {0};
\node(55) at (5,0)[shape=circle,draw,scale=.7]        {0};
\node(6) at (2.8,0) {};
\node(7) at (3.2,0) {};
\path [-] (3)    edge   node     {}     (4);
\path [-] (5)    edge   node     {}     (55);
\path [-] (4)      edge   node [above]  {}     (6);
\path [-] (7)      edge   node [above]  {}     (5);
\node(8) at (3,0)        {$\cdots$};

\node(10) at (2,.6) {};
\path [-] (10)    edge   node     {}     (4);
\node(100) at (4,.6) {};
\path [-] (100)    edge   node     {}     (5);
\node(11) at (5.5,.6) {};
\path [-] (11)    edge   node     {}     (55);
\node(12) at (5.5,-.6) {};
\path [-] (12)    edge   node     {}     (55);

\node(21) at (.3,.5) {};
\path [-] (21)    edge   node     {}     (3);
\node(22) at (.3,-.5) {};
\path [-] (22)    edge   node     {}     (3);
\node at (.5,.1) {$\vdots$};

\node(1) at (-.2,0)        {$_{n-1}$};
\node(2) at (3.45,-.55)        {$_k$};
\draw [decorate,
    decoration = {brace}](5.1,-.3)--(1.8,-.3);

\draw [decorate,
    decoration = {brace}](.2,-.4)--(.2,.4);
\end{tikzpicture}
\end{center}
which defines the stratum:
\[ \overline{\modm}_\Gamma=\overline{\modm}_{g,n}\times \overbrace{\modm_{0,3}\times...\times\modm_{0,3}}^{k\text{ factors}}\stackrel{\phi_\Gamma}{\longrightarrow}\overline{\modm}_{g,n+k}.
\]
For $m=2g-2+n+k>2g-2+n$,
\[0=\Omega_{g,n}\otimes\Omega_{0,3}^{\otimes k}= \phi_\Gamma^*\Omega_{g,n+k}=\phi_\Gamma^*K_m=K_m\otimes 1^{\otimes k} \]
where the first equality uses $\Omega_{0,3}=0$ which vanishes for degree reasons, the second equality uses the assumption that
$\Omega_{g,n}=K_{2g-2+n}$ is a CohFT without unit, and the final equality uses the degree $m$ part of the restriction
\[\phi_\Gamma^*\bk=\phi_\Gamma^*\exp(\sum s_i\kappa_i)=\exp(\sum s_i\kappa_i)\otimes \exp(\sum s_i\kappa_i)^{\otimes k}=\exp(\sum s_i\kappa_i)\otimes 1^{\otimes k}.\]
Hence $K_m=0$ for $m>2g-2+n$ and $n>0$.  When $n=0$, for $2g-2<m<3g-3$ from the push-forward formula \eqref{pushf} we have
\[ K_m=\frac{1}{6g-6-2m}\pi_*K_{m+1}=0\]
since we have shown that $K_{m+1}=0$ in $H^*(\overline{\modm}_{g,1},\bq)$.  Although $K_{3g-2}=0$ in $H^*(\overline{\modm}_{g,1},\bq)$, the push-forward formula \eqref{pushf} $\pi_*K_{3g-2}=0\times K_{3g-3}$ gives no information about $K_{3g-3}\in H^*(\overline{\modm}_{g},\bq)$, and in particular $K_{3g-3}$ can be non-vanishing.\\

For the case $\Omega_{g,n}=J_{2g-2+n}$ the argument is exactly the same as for $K_{2g-2+n}$ except for the push-forward formula \eqref{pushf}  in the case $n=0$,
\[ J_m=\frac{1}{2g-2-m}\pi_*J_{m+1}=0\]
which proves that $J_m=0$ in $H^*(\overline{\modm}_{g},\bq)$ for $m>2g-2$ (including $m=3g-3$).  Hence
$$J_m(\kappa_1,...,\kappa_m)=0\quad\text{for}\quad \left\{\begin{array}{ll}m>2g-2+n&\\
m=2g-2+n,&n>1.\end{array}\right.$$
The extra vanishing $J_{2g-2+n}=0$ for $n>1$ follows from the assumption $\Omega'_{g,n}=0$ for $n>1$.
\end{proof}

The push-forward relations \eqref{pushf} have an interpretation in terms of a shift of 1-dimensional CohFTs.  Given a 1-dimensional CohFT $\Omega_{g,n}\in H^*(\overline{\modm}_{g,n},\bq)$ define its shift depending on a parameter $\epsilon$ by:
$$\Omega_{g,n}(\epsilon)=\sum_{m\geq 0}  \pi_*^m\Omega_{g,n+m}\cdot\frac{ \epsilon^m}{m!}.
$$
This defines a CohFT for values $\epsilon$ where the sum converges.
\begin{lemma}
$$\bj(\epsilon):=\sum_{m\geq 0}  \pi_*^m \bj\cdot\frac{ \epsilon^m}{m!}=\sum d_k(\epsilon)J_kt^k,\quad d_k(\epsilon)=e^{-(k+\chi)\epsilon}
$$
$$\bk(\epsilon):=(1-\epsilon)^{2g-2+n}\sum_{m\geq 0}  \pi_*^m \bk\cdot\frac{ \epsilon^m}{m!}=\sum c_k(\epsilon)K_kt^k,\quad c_k(\epsilon)=(1-\epsilon)^{2(k+\chi)}.
$$
\end{lemma}
The factor $(1-t)^{2g-2+n}$ is chosen in the definition of $\bk(\epsilon)$ for convenience.   Note that any factor of $\lambda^{2g-2+n}$ sends a CohFT to a CohFT.
\begin{proof}
Iterating \eqref{pushf}, we have
$$\pi_*^mJ_k=M^mJ_{k-m},\quad M=2g-2+n-k.
$$
Hence the coefficient of $J_kt^k$ in $\sum_{m\geq 0}  \pi_*^m \bj\cdot\frac{ \epsilon^m}{m!}$ is $e^{M\epsilon}=e^{-(k+\chi)\epsilon}$ as required.

For $K_k$, we iterate \eqref{pushf} to get
$$\pi_*^mK_k=M(M-1)...(M-m+1)K_{k-m},\quad M=3(2g-2+n)-2k-1.
$$
Hence the coefficient of $K_kt^k$ in $\sum_{m\geq 0}  \pi_*^m \bk\cdot\frac{ \epsilon^m}{m!}$ is $\displaystyle\sum_{k\geq 0}\binom{M+k}{k}\epsilon^k=(1-\epsilon)^{-M-1}$, for $|\epsilon|<1$.  When we multiply by $(1-t)^{2g-2+n}$, the coefficient of $K_kt^k$ in $\bk(\epsilon)$ is $(1-\epsilon)^{2(k+\chi)}$ as required.
\end{proof}
The shift is equivalent to rescaling $(a,b)$ in Definition~\ref{reciprocal}.  Then $\bj(\epsilon)$ corresponds to $(a,b)=(e^{-\epsilon},e^{-\epsilon})$ and $\bk(\epsilon)$ corresponds to $(a,b)=(3(1-\epsilon)^2,2(1-\epsilon)^2)$.  We see that $\bj(\epsilon)$ is defined for all $\epsilon\in\bc$ and $\bk(\epsilon)$ is defined for all $|\epsilon|<1$ and can be analytically continued to $\epsilon\in\bc-\{1\}$.  We can restate Conjecture~\ref{conj} using $\bj(\epsilon)$ and $\bk(\epsilon)$:
$$\text{Conjecture }\ref{conj}\Leftrightarrow\lim_{\epsilon\to-\infty}\bj(\epsilon)\text{ exists\quad and}\quad \lim_{\epsilon\to 1}\bk(\epsilon)\text{ exists}.
$$

\section{Geometric constructions}  \label{sec:geom}
In this section we give three geometric constructions which may give a way to prove the vanishing results involving the classes $\bj$ and $\bk$. 

\subsection{Hodge CohFT}  \label{hodge}
The Hodge bundle $\be_g\to\overline{\modm}_{g,n}$ satisfies natural restriction properties on the boundary:
\begin{equation}  \label{hodres}
\begin{array}{l}0\to\be_{g-1}\to\phi_{\text{irr}}^*\be_g\to\co_{\overline{\modm}_{g-1,n+2}}\to 0\\
\\
0\to\pi_1^*\be_{h}\oplus\pi_2^*\be_{g-h}\to \phi_{h,I}^*\be_g\to 0
\end{array}
\end{equation}
where $\pi_i$ is projection onto the $i$th factor of $\overline{\modm}_{h,|I|+1}\times\overline{\modm}_{g-h,|J|+1}$.  A consequence of the restriction properties is that the total Chern class $c(\be_g)$ of the Hodge bundle defines a CohFT.   The relation $\pi^*\be_g=\be_g$ implies that the CohFT has flat unit.

Similarly, the bundle $\be_g\oplus\be_g$ satisfies natural restriction properties and its total Chern class $c(\be_g)^2=c(\be_g\oplus\be_g)$ defines a CohFT.  Alternatively, the square of any CohFT satisfies the properties of a CohFT.  Define:
\[\Omega^{\text{Hod}2}_{g,n}:=c(\be)^2=(1+\lambda_1+...+\lambda_g)^2.
\]
Flatness of $\be\oplus\be^\vee$ implies $c(\be\oplus\be^\vee)=c(\be)c(\be^\vee)=1$ hence $\lambda_g^2=0$ and $\lambda_{g-1}^2=2\lambda_{g-2}\lambda_g$ so
\begin{align*}
\Omega^{\text{Hod}2}_{g,n}&=\lambda_g^2+2\lambda_{g-1}\lambda_g+\lambda_{g-1}^2+2\lambda_{g-2}\lambda_g+...\\
&=(-1)^{g+n}2^{2-n}\Lambda_{g,n}+\deg_{<2g-2+n}\text{ terms}
\end{align*}
where
\begin{equation} \label{Lambda}
\Lambda_{g,n}=\left\{\begin{array}{cc}(-1)^g\lambda_{g-2}\lambda_g&n=0\\(-1)^{g-1}\lambda_{g-1}\lambda_g&n=1\\0&n>1\end{array}\right.
\end{equation}

\begin{proposition}
The family of classes $\{\Lambda_{g,n}\}$ defines a CohFT without unit, i.e. they satisfy the pull-back properties \eqref{glue1} and \eqref{glue2}.
\end{proposition}
\begin{proof}
We can prove this via the relation with the degree $2g-2+n$ terms of $\Omega^{\text{Hod}2}_{g,n}$, or directly as follows.
The Hodge classes satisfy the restriction property:
\[ \phi_{\text{irr}}^*\lambda_k=\lambda_k,\quad\phi_{h,I}^*\lambda_k=\sum_{a+b=k}\lambda_a\otimes\lambda_b.
\]
Thus $\phi_{\text{irr}}^*\lambda_g=0$ since $\lambda_g=0$ in $H^*(\overline{\modm}_{g-1,n+2},\bq)$ and $\phi_{h,I}^*\lambda_g=\lambda_{h}\otimes\lambda_{g-h}$ again since higher degree $\lambda_k=0$ for $k>h$ in the first factor and $k>g-h$ in the second factor.  The same argument gives $\phi_{h,I}^*\lambda_{g-1}=\lambda_{h}\otimes\lambda_{g-h-1}+\lambda_{h-1}\otimes\lambda_{g-h}$.  Thus
\[\phi_{\text{irr}}^*\lambda_g\lambda_{g-1}=0,\quad \phi_{h,I}^*\lambda_g\lambda_{g-1}=\lambda_{h}^2\otimes\lambda_{g-h}\lambda_{g-h-1}+\lambda_{h}\lambda_{h-1}\otimes\lambda_{g-h}^2.
\]
We have $c(\be\oplus\be^\vee)=1$ for $\be$ the Hodge bundle and $\be^\vee$ its dual, so in particular $\lambda_g^2=0$ for $g>0$.  Thus  $\phi_{h,I}^*\lambda_g\lambda_{g-1}=0$ for $h(g-h)>0$  leaving only $\phi_{0,I}^*\lambda_g\lambda_{g-1}=1\otimes\lambda_{g}\lambda_{g-1}$.  On $\overline{\modm}_{g,1}$, all boundary divisors have positive genus irreducible components, so we conclude that the pullback of $\lambda_g\lambda_{g-1}\in H^*(\overline{\modm}_{g,1},\bq)$ to the boundary is 0.  Hence we have proven that
\[\phi^*_\Gamma\Lambda_{g,1}=0\]
as required since any boundary divisor $\overline{\modm}_\Gamma\subset\overline{\modm}_{g,1}$ involves a component with $n>1$.
\\

For the pull-back $\phi^*_\Gamma\Lambda_{g,0}$ we have $\phi_{\text{irr}}^*\lambda_{g}=0$ hence $\phi_{\text{irr}}^*\Lambda_{g}=0$ as required on $H^*(\overline{\modm}_{g-1,2},\bq)$.  Finally,
\begin{align*}
\phi_{h,I}^*\lambda_{g-2}\lambda_{g}&=\lambda_{h-2}\lambda_{h}\otimes\lambda_{g-h}^2+\lambda_{h-1}\lambda_{h}\otimes\lambda_{g-1}\lambda_{g}+\lambda_{h}^2\otimes\lambda_{g-2}\lambda_{g}\\
&=\lambda_{h-1}\lambda_{h}\otimes\lambda_{g-h-1}\lambda_{g-h}
\end{align*}
and $(-1)^{h-1}(-1)^{g-h-1}=(-1)^g$ shows that the signs agree yielding
\[\phi_{h,I}^*\Lambda_{g}=\Lambda_{h,1}\otimes \Lambda_{g-h,1}.
\]
\end{proof}

The interest in $\Lambda_{g,n}$ is due to the equality checked via admcycles
\[ J_{2g-2+n}=\Lambda_{g,n},\quad 2g-2+n\leq 6.
\]
We expect this equality to hold for all $g$ and $n$ in which case Proposition~\ref{equivconj} would apply to conclude \eqref{Jrel} in Conjecture~\ref{conj}.
\begin{proposition}  \label{JHodge}
If
\begin{equation}  \label{assumption}
J_{2g-2+n}=\left\{\begin{array}{cc}(-1)^{g}\lambda_{g-2}\lambda_g&n=0\\(-1)^{g-1}\lambda_{g-1}\lambda_g&n=1\end{array}\right.
\end{equation}
then
$$J_m(\kappa_1,...,\kappa_m)=0\quad\text{for}\quad \left\{\begin{array}{ll}m>2g-2+n&\\
m=2g-2+n,&n>1.\end{array}\right.
$$
\end{proposition}
\begin{proof}
For $\Lambda_{g,n}$ defined in \eqref{Lambda}, the more general equality
\[ J_{2g-2+n}=\Lambda_{g,n},\quad 2g-2+n>0
\]
follows from the initial cases of $n=0,1$ given by \eqref{assumption}.  Equivalently, \eqref{assumption} implies that $J_{2g-2+n}=0$ for $n>1$, which we now show.  For even $n=2m>0$ consider the stable graph $\Gamma\in G_{g+m}$ with a single vertex and $m$ loops which defines the map
\[\phi_\Gamma:\overline{\modm}_{g,n}\to\overline{\modm}_{g+m}
\]
by identifying the labeled points $p_i$, say $p_1\sim p_2$, $p_3\sim p_4$, ..., $p_{n-1}\sim p_{n}$ to produce a stable curve of genus $g+m$. Then in $H^*(\overline{\modm}_{g,n},\bq)$
\[ J_{2g-2+n}=\phi_\Gamma^*J_{2(g+m)-2}=\phi_\Gamma^*(-1)^{g+m}\lambda_{g+m-2}\lambda_{g+m}=0
\]
where the first equality uses the degree $2g-2+n$ part of the restriction
\[\phi^*\bj=\phi_\Gamma^*\exp(\sum \sigma_i\kappa_i)=\exp(\sum \sigma_i\kappa_i)=\bj,\]
the second equality uses  $J_{2(g+m)-2}=(-1)^{g+m}\lambda_{g+m-2}\lambda_{g+m}$ which is the assumption \eqref{assumption}, and the third equality uses $0=\phi_\Gamma^*\lambda_{g+m}\in H^*(\overline{\modm}_{g,n},\bq)$ since $g+m>g$.

For odd $n=2m+1>1$ the argument is similar.  The stable graph $\Gamma\in G_{g+m,1}$ with a single vertex, $m$ loops and a leaf defines
\[\phi_\Gamma:\overline{\modm}_{g,n}\to\overline{\modm}_{g+m,1}
\]
by identifying the labeled points $p_1\sim p_2$, $p_3\sim p_4$, ..., $p_{n-2}\sim p_{n-1}$.  Then in $H^*(\overline{\modm}_{g,n},\bq)$
\[ J_{2g-2+n}=\phi_\Gamma^*J_{2(g+m)-1}=\phi_\Gamma^*(-1)^{g+m-1}\lambda_{g+m-1}\lambda_{g+m}=0
\]
as above.

Thus, the assumption \eqref{assumption} implies that $J_{2g-2+n}=\Lambda_{g,n}$ is a CohFT without unit and $J_{2g-2+n}=0$ for $n>1$.  Hence Proposition~\ref{equivconj} applies to deduce the vanishing of $J$ polynomials in the range given by  Conjecture~\ref{conj}.
\end{proof}

\begin{proposition}  \label{Kpsi}
For $g>1$ and $n>0$, if
\begin{equation}  \label{Kassumption}
K_{2g-2+n}=\left(\prod_{i=1}^n\psi_i\right)\pi^*K_{2g-2}
\end{equation}
then
$K_m(\kappa_1,...,\kappa_m)=0$ for $m>2g-2+n$ except $(m,n)=(3g-3,0)$.
\end{proposition}
\begin{proof}
The proof is similar to the proof of Proposition~\ref{equivconj}.  For  $k>0$ and $n>0$,  consider the stable graph $\Gamma\in G_{g,n+k}$ given by
\begin{center}
\begin{tikzpicture}
\node(1) at (-1,0)        {$\Gamma=$};
\node(3) at (1,0)[shape=circle,draw,scale=.7]        {$g$};
\node(4) at (2,0)[shape=circle,draw,scale=.7]        {0};
\node(5) at (4,0)[shape=circle,draw,scale=.7]        {0};
\node(55) at (5,0)[shape=circle,draw,scale=.7]        {0};
\node(6) at (2.8,0) {};
\node(7) at (3.2,0) {};
\path [-] (3)    edge   node     {}     (4);
\path [-] (5)    edge   node     {}     (55);
\path [-] (4)      edge   node [above]  {}     (6);
\path [-] (7)      edge   node [above]  {}     (5);
\node(8) at (3,0)        {$\cdots$};

\node(10) at (2,.6) {};
\path [-] (10)    edge   node     {}     (4);
\node(100) at (4,.6) {};
\path [-] (100)    edge   node     {}     (5);
\node(11) at (5.5,.6) {};
\path [-] (11)    edge   node     {}     (55);
\node(12) at (5.5,-.6) {};
\path [-] (12)    edge   node     {}     (55);

\node(21) at (.3,.5) {};
\path [-] (21)    edge   node     {}     (3);
\node(22) at (.3,-.5) {};
\path [-] (22)    edge   node     {}     (3);
\node at (.5,.1) {$\vdots$};

\node(1) at (-.2,0)        {$_{n-1}$};
\node(2) at (3.45,-.55)        {$_k$};
\draw [decorate,
    decoration = {brace}](5.1,-.3)--(1.8,-.3);

\draw [decorate,
    decoration = {brace}](.2,-.4)--(.2,.4);
\end{tikzpicture}
\end{center}
and its associated stratum:
\[ \overline{\modm}_\Gamma=\overline{\modm}_{g,n}\times \overbrace{\modm_{0,3}\times...\times\modm_{0,3}}^{k\text{ factors}}\stackrel{\phi_\Gamma}{\longrightarrow}\overline{\modm}_{g,n+k}.
\]
Then
\begin{equation}  \label{van}
0= \phi_\Gamma^*K_{2g-2+n+k}=K_{2g-2+n+k}\otimes 1^{\otimes k}
\end{equation}
where the first equality uses the assumption $K_{2g-2+n+k}=\left(\prod_{i=1}^{n+k}\psi_i\right)\pi^*K_{2g-2}$, so in particular there is a $\psi_i$ which vanishes on the corresponding factor of $\modm_{0,3}$.  The second equality uses the degree $2g-2+n+k$ part of the restriction
\[\phi_\Gamma^*\bk=\phi_\Gamma^*\exp(\sum s_i\kappa_i)=\exp(\sum s_i\kappa_i)\otimes \exp(\sum s_i\kappa_i)^{\otimes k}=\exp(\sum s_i\kappa_i)\otimes 1^{\otimes k}.\]
Hence $K_m=0$ for $m>2g-2+n$ and $n>0$.  The $n=0$ cases uses the $n=1$ case and the push-forward formula when $2g-2<m<3g-3$
\[ K_m=\frac{1}{6g-6-2m}\pi_*K_{m+1}=0.\]
\end{proof}
\begin{remark}  \label{Kweak}
By essentially the same proof, a weak version of Proposition~\ref{Kpsi} also holds.  If
\begin{equation}  \label{pullbweak}
\int_{\overline{\modm}_{g,n}}\Big(K_{2g-2+n}-\Big(\prod_{i=1}^n\psi_i\Big)\pi^*K_{2g-2}\Big)\cdot\omega=0
\end{equation}
for all tautological classes $\omega\in RH^*(\overline{\modm}_{g,n})$ then  $\int_{\overline{\modm}_{g,n}}K_m\cdot\omega=0$ for all $\omega\in RH^*(\overline{\modm}_{g,n})$ and $m>2g-2+n$ except $(m,n)=(3g-3,0)$.  In the proof, simply take the intersection of \eqref{van} with $(\phi_\Gamma)_*(\omega\otimes 1^{\otimes k})$ for any tautological class $\omega\in RH^*(\overline{\modm}_{g,n})$ .
\end{remark}

\subsection{Spin CohFT}   \label{spincohft}
We define here a two dimensional CohFT arising out of the total Chern class of a natural bundle defined over the moduli space of spin curves, analogous to the Hodge bundle and its related CohFT.  Define the moduli space of stable twisted spin curves \cite{AJaMod} by
$$\overline{\modm}_{g,n}^{\rm spin}=\{(\cc,\theta,p_1,...,p_n,\phi)\mid \phi:\theta^2\stackrel{\cong}{\longrightarrow}\omega_{\cc}^{\text{log}}\}.
$$
The pair $(\theta,\phi)$ is a {\em spin structure} on $\cc$ where $\omega_{\cc}^{\text{log}}$ and $\theta$ are line bundles over the stable twisted curve $\cc$ with labeled orbifold points $p_j$ and $\deg\theta=g-1+\frac12n$.     A line bundle $L$ over a twisted curve $\cc$ is a locally equivariant bundle over the local charts, such that at each nodal point there is an equivariant isomorphism of fibres.  Hence each orbifold point $p$ associates a representation of $\bz_2$ on $L|_p$ acting by multiplication by $\exp(\pi i\sigma_p)$ for $\sigma_p\in\{0,1\}$.  The equivariant isomorphism of fibres over nodal points forces a balanced condition at each node $p$ given by $\sigma_{p_+}=\sigma_{p_-}$ for $p_\pm$ corresponding to the node on each irreducible component.

Define a vector bundle over $\overline{\modm}_{g,n}^{\rm spin}$ using the dual bundle $\theta^{\vee}$ on each stable twisted curve as follows.  Denote by $\ce$ the universal spin structure on the universal stable twisted spin curve over $\overline{\modm}_{g,n}^{\rm spin}$.  Given a map $S\to\overline{\modm}_{g,n}^{\rm spin}$, $\ce$ pulls back to $\theta$ giving a family $(\cc,\theta,p_1,...,p_n,\phi)$ where $\pi:\cc\to S$ has stable twisted curve fibres, $p_i:S\to\cc$ are sections with orbifold isotropy $\bz_2$ and $\phi:\theta^2\stackrel{\cong}{\longrightarrow}\omega_{\cc/S}^{\text{log}}=\omega_{\cc/S}(p_1,..,p_n)$.  The derived push-forward sheaf of $\ce^{\vee}$ over $\overline{\modm}_{g,n}^{\rm spin}$ gives rise to a bundle as follows.  The space of global sections of $\pi_*\ce^{\vee}$ vanishes, i.e. $R^0\pi_*\ce^{\vee}=0$, since
\[\deg\theta^{\vee}=1-g-\frac12n<0
\]
and furthermore, for any irreducible component $\cc'\stackrel{i}{\to}\cc$,
\[\deg\theta^{\vee}|_{\cc'}<0
\]
since  $i^*\omega_{\cc/S}^{\text{log}}=\omega_{\cc'/S}^{\text{log}}$ and $\cc'$ is stable  guarantees its log canonical bundle has negative degree.  Thus $R^1\pi_*\ce^\vee$ defines the following bundle.
\begin{definition}  \label{obsbun}
Define a bundle $E_{g,n}=-R\pi_*\ce^\vee$ over $\overline{\modm}_{g,n}^{\rm spin}$ with fibre $H^1(\theta^{\vee})$.
\end{definition}

The moduli space of stable twisted spin curves decomposes into components labeled by $\vec{\sigma}\in\{0,1\}^n$ satisfying $\displaystyle|\vec{\sigma}|+n\in 2\bz$:
$$\overline{\modm}_{g,n}^{\rm spin}=\bigsqcup_{\vec{\sigma}}\overline{\modm}_{g,n,\vec{\sigma}}^{\rm spin}
$$
where $\overline{\modm}_{g,n,\vec{\sigma}}^{\rm spin}$ consists of spin curves with induced representation at labeled points determined by $\vec{\sigma}$ as follows.
 The representations induced by $\theta$ at the labeled points produce a vector $\vec{\sigma}=(\sigma_1,...,\sigma_n)\in\{0,1\}^n$ which determine the connected component.  The number of $p_i$ with $\lambda_{p_i}=0$ is even due to evenness of the degree of the push-forward sheaf $|\theta|:=\rho_*\co_\cc(\theta)$ on the coarse curve $C$, \cite{JKVMod}.  In the smooth case, the
boundary type of a spin structure is determined by an associated quadratic form applied to each of the $n$ boundary classes which vanishes since it is a homological invariant, again implying that the number of $p_i$ with $\lambda_{p_i}=0$ is even.
Each component $\overline{\modm}_{g,n,\vec{\sigma}}^{\rm spin}$ is connected except when $|\vec{\sigma}|=n$, in which case there are two connected components determined by their Arf invariant, and known as even and odd spin structures.  This follows from the case of smooth spin curves proven in \cite{NatMod}.

Restricted to $\overline{\modm}_{g,n,\vec{\sigma}}^{\rm spin}$, the bundle $E_{g,n}$ has rank
\begin{equation}  \label{rank}
\text{rank\ }E_{g,n}=2g-2+\tfrac12(n+|\vec{\sigma}|)
\end{equation}
by the following Riemann-Roch calculation.   Orbifold Riemann-Roch takes into account the representation information \begin{align*}
h^0(\theta^{\vee})-h^1(\theta^{\vee})&=1-g+\deg \theta^{\vee}-\sum_{i=1}^n\lambda_{p_i}=1-g+1-g-\tfrac12n-\tfrac12|\vec{\sigma}|\\
&=2-2g-\tfrac12(n+|\vec{\sigma}|).
\end{align*}
We have $h^0(\theta^{\vee})=0$ since $\deg\theta^{\vee}=1-g-\frac12n<0$, and the restriction of $\theta^{\vee}$ to any irreducible component $C'$, say of type $(g',n')$, also has negative degree, $\deg\theta^{\vee}|_{C'}=1-g'-\frac12n'<0$.  Hence $h^1(\theta^{\vee})=2g-2+\frac12(n+|\vec{\sigma}|)$.  Thus $H^1(\theta^{\vee})$ gives fibres of a rank $2g-2+\frac12(n+|\vec{\sigma}|)$ vector bundle.


The forgetful map $p:\overline{\modm}_{g,n}^{\rm spin}\to\overline{\modm}_{g,n}$ and its restriction $p^{\bar{\sigma}}\hspace{-1mm}:\overline{\modm}_{g,n}^{\rm spin}\to\overline{\modm}_{g,n,\bar{\sigma}}$ forgets the spin structure and maps the twisted curve to its underlying coarse curve.
\begin{definition}
On $\bc^2=\langle e_0,e_1\rangle$ with metric $\eta(e_i,e_j)=\frac12\delta_{ij}$ define the CohFT
$$\Omega^{\text{spin}}_{g,n}(e_{\sigma_1}\otimes...\otimes e_{\sigma_n}):=p^{\bar{\sigma}}_*c(E_{g,n})\in H^{*}(\overline{\modm}_{g,n},\bq).
$$
\end{definition}
The different components of the moduli space behave like the different components of the moduli space of stable maps used in the CohFT associated to Gromov-Witten invariants.  The proof that $\{\Omega^{\text{spin}}_{g,n}\}$ defines a CohFT is a special case of a more general result  in \cite{LPSZChi} which proves that push-forwards of classes defined over moduli spaces of $r$-spin curves known as Chiodo classes produce dimension $r$ Coh\-FTs.  The proof in  \cite{LPSZChi} uses the Givental action on CohFTs, or instead one can use natural restriction properties of the bundles $E_{g,n}$ analogous to \eqref{hodres} proven in \cite{NorNew}.

The CohFT vanishes in degrees greater than $2g-2+n$ = the largest rank of $E_{g,n}$ on components of $\overline{\modm}_{g,n}^{\rm spin}$, hence its degree $2g-2+n$ terms define a CohFT without unit:
$$\Theta_{g,n}:=(-1)^n2^{g-1+n}p_*c_{2g-2+n}(E_{g,n})\in H^{4g-4+2n}.
$$
For example, $\Theta_{1,1}=3\psi_1$ and $\Theta_{0,n}=0$ for all $n$.
The factor of 2 in the bivector $\Delta=2(e_0\otimes e_0+e_1\otimes e_1)$ leads to the rescaling by a power of 2, which is quite naturally given by $2^{1-g}$, and we choose the further rescaling by $(-2)^{2g-2+n}$ to produce the elegant formula involving the Br\'ezin-Gross-Witten tau function of the KdV hierarchy in Proposition~\ref{bgwtheta} below.
Define the partition function
\[
Z^\Theta(\hbar,t_0,t_1,...)=\exp\sum_{g,n,\vec{k}}\frac{\hbar^{g-1}}{n!}\int_{\overline{\modm}_{g,n}}\Theta_{g,n}\cdot\prod_{j=1}^n\psi_j^{k_j}\prod t_{k_j}.
\]

\begin{proposition}  \label{bgwtheta}
\[Z^{\text{BGW}}(\hbar,t_0,t_1,...)=Z^\Theta(\hbar,t_0,t_1,...)+O(\hbar^7).\]
\end{proposition}
\begin{remark}
The $O(\hbar^7)$ term is conjectured in \cite{NorNew} to vanish which is a consequence of Conjecture~\ref{KTheta}, that $K_{2g-2+n}=\Theta_{g,n}$.
\end{remark}
\begin{proof}
The strategy of proof is to prove equality of intersection numbers given by $Z^\Theta(\hbar,t_0,t_1,...)=e^{\hbar^{-1}F^K}$ and then to use Theorem~\ref{limBGW}.  Such intersection

1. Reduce the equality to finitely many terms for each genus $g$ as follows.  It is proven in \cite{NorNew} that the classes $\Theta_{g,n}$ satisfy the pull-back property
\[
\Theta_{g,n+1}=\psi_{n+1}\cdot\pi^*\Theta_{g,n}
\]
which implies that $Z^\Theta(\hbar,t_0,t_1,...)$ satisfies the following dilaton equation.
\begin{equation}  \label{dilaton}
\frac{\partial}{\partial t_0}Z^{\Theta}(\hbar,t_0,t_1,...)=\sum_{i=0}^{\infty}(2i+1)t_i \frac{\partial}{\partial t_i}Z^{\Theta}(\hbar,t_0,t_1,...)+\frac18Z^{\Theta}(\hbar,t_0,t_1,...)
\end{equation}
A consequence is that the intersection numbers $\int_{\overline{\modm}_{g,n}}\Theta_{g,n}\cdot\prod_{j=1}^n\psi_j^{k_j}$ for $n\geq g$ are uniquely determined by those for $n<g$.  Moreover, $Z^{\text{BGW}}(\hbar,t_0,t_1,...)$ also satisfies \eqref{dilaton} hence equality of $\log Z^{\Theta}$ and $\log Z^{\text{BGW}}$ for a given genus $g$ follows from the finitely many intersection numbers for $(g,n)$ where $n<g$.

2. {\em Assumption: \eqref{pullbweak} holds for all $(g,n)$ with $g\leq 7$ and $n\leq g-1$.}  Using the software admcycles we can only directly calculate terms on the RHS up to genus 4 so we instead use a weak argument as described in Remark~\ref{Kweak}.  Given the assumption above, let $\omega\in RH^{g-1}(\overline{\modm}_{g,n})$ be a polynomial in $\kappa$ classes and $\psi$ classes for $g\leq 7$ and $n\leq g-1$.  Then for $g>1$
\begin{align*}
\int_{\overline{\modm}_{g,n}}K_{2g-2+n}\cdot\omega&=\int_{\overline{\modm}_{g,n}}\psi_n\pi^*K_{2g-2+n-1}\cdot\omega=\int_{\overline{\modm}_{g,n-1}}K_{2g-2+n-1}\pi_*(\psi_n\omega_n)\\
&=\int_{\overline{\modm}_{g,n-1}}K_{2g-2+n-1}\omega_{n-1}=\int_{\overline{\modm}_{g}}\pi^*K_{2g-2}\cdot\omega_0(\kappa_1,\kappa_2,...,\kappa_{g-1})
\end{align*}
where the first equality used the assumption \eqref{pullbweak}.  Also, $\omega_n:=\omega$ and inductively $\pi_*(\psi_k\omega_k)=\omega_{k-1}\in RH^{g-1}(\overline{\modm}_{g,k-1})$ is a polynomial in $\kappa$ classes and $\psi$ classes.  This uses the inductive hypothesis that $\omega_{k}$ is a polynomial in $\kappa$ classes and $\psi$ classes, together with the property for $j<k$, $\psi_k\psi_j=\psi_k\pi^*\psi_j\in RH^{*}(\overline{\modm}_{g,k-1})$ and push-forward properties of $\kappa$ classes obtained by replacing $\kappa_j$ by $\kappa_{\ell_j}=\pi^*\kappa_j+\psi_n^j$ and push-forward all summands.  When $g=1$, the right hand side is instead $\int_{\overline{\modm}_{1,1}}K_1\cdot p$ for $p\in\bq$ a constant.  The same inductive argument proves for $g>1$
$$\int_{\overline{\modm}_{g,n}}\Theta_{g,n}\cdot\omega=\int_{\overline{\modm}_{g}}\Theta_{g}\cdot \omega_0(\kappa_1,\kappa_2,...,\kappa_{g-1})$$
and again when $g=1$, the right hand side is $\int_{\overline{\modm}_{1,1}}\Theta_{1,1}\cdot p$ for $p\in\bq$ a constant.

By a result of Faber and Pandharipande \cite[Proposition 2]{FPaRel}, a homogeneous degree $g-1$ monomial in the $\kappa$ classes is equal in the tautological ring to the sum of boundary terms, i.e. the sum of push-forwards of polynomials in $\psi$ classes and $\kappa$ classes by the the maps $(\phi_\Gamma)_*$.    But the pull-back of $K_{2g-2}$ and $\Theta_g$ to any boundary divisor is a tensor product of $K_{2g'-2+n'}$ terms, respectively $\Theta_{g',n'}$ terms, for $g'<g$.  The pull-back of $K_{2g-2}$ uses the pull-back of the exponential $\bk$ together with the weak vanishing of $K_m$ for $m>2g'-2+n'$.  By induction, $K_{2g-2+n}$ and $\Theta_{g,n}$ agree weakly, for lower genus, hence
\[ \int_{\overline{\modm}_{g,n}}(K_{2g-2+n}-\Theta_{g,n})\cdot\omega=0
\]
for any polynomial $\omega$ in $\psi$ classes and $\kappa$ classes (and in fact any tautological class) so in particular any polynomial in $\psi$ classes.  By Theorem~\ref{limBGW} this proves the required result under assumption 1 above, for $g\leq 7$ and $n\leq g-1$ and hence by the reduction in 1 above, for $g\leq 7$ and $n\geq 0$.

It remains to prove the assumption that \eqref{pullbweak} holds for all $(g,n)$ with $g\leq 7$ and $n\leq g-1$.  This is a relation between intersection numbers involving only polynomials in $\psi$ classes and $\kappa$ classes which have been well-studied.  It is achieved using different software, in particular by admcycles for $g\leq 7$ and $n\leq g-1$, hence the proposition follows.
\end{proof}

\subsection{Monotone Hurwitz numbers }   \label{monot}

The class $\bk=1+K_1+K_2+\dots$ of Theorem~\ref{numeric} naturally arises also in the study of monotone Hurwitz numbers. The \emph{monotone Hurwitz  numbers} $\vec h_{g,(k_1,\dots,k_n)}$ enumerate factorizations in the symmetric group $S(\sum k_i)$ into transpositions of the form
$$
(a_1,b_1)\circ(a_2,b_2)\circ\dots\circ(a_m,b_m),\quad a_i<b_i,
$$
such that the product has cyclic type $(k_1,\dots,k_n)$, the integers $m$ and $g$ are related by the Riemann-Hurwitz relation $m=2g-2+n+\sum k_i$, the subgroup generated by transpositions acts transitively on the set of permuted elements, and also the following additional monotonicity condition is satisfied: the bigger numbers $b_i$ of the pairs of transposed elements satisfy the inequality
$b_1\le b_2\le\dots\le b_m$.

For relations of the present section it would be more convenient to consider a slightly rescaled version of the class $\bk$, namely, define 
\[\bk^*=1-K_1+K_2-K_3+\dots=\exp\big(\sum (-1)^is_i\kappa_i\big)\] 
where $s_i$ are as in Section~\ref{sec:constr}. We have proved in Section~\ref{sec:van} that the differentials $\omega_{g,n}$ of topological recursion with the rational spectral curve
$$
x=\frac{z^2}{2},\qquad y=\frac{z}{1-z^2}.
$$
are given explicitly for $2g-2+n>0$ by
\begin{equation}\label{intnum}
w_{g,n}(z_1,\dots,z_n)=\sum_{k_1,\dots,k_n}\int_{\overline{\modm}_{g,n}}
\bk^*\;\psi_1^{k_1}\dots\psi_n^{k_n}~\prod_{i=1}^n\frac{(2k_i+1)!!dz_i}{z^{2k_i+2}_i}
\end{equation}
(it corresponds to the specialization $\epsilon=-1$ in~\eqref{yfuncK}).

It turns out that the monotone Hurwitz numbers are governed by the very same topological recursion with a special choice of the expansion point $z=1$ and the local coordinate $X$ at this point defined by the change $z=\sqrt{1-4X}$.

\begin{proposition}[\cite{DDMTop}] In the stable range, that is, for each $(g,n)$ with $2g-2+n>0$, the power expansions of the above differential $\omega_{g,n}(z_1,\dots,z_n)$ in the local coordinates $X_1,\dots, X_n$ obtained by the substitution $z_i=\sqrt{1-4 X_i}$ are given by
$$
\omega_{g,n}=\sum_{k_1,\dots,k_n}\vec h_{g,(k_1,\dots,k_n)}\prod_{i=1}^n k_iX_i^{k_i-1}dX_i.
$$
\end{proposition}

\begin{proof}
In fact, a slightly different but equivalent form of topological recursion for the monotone Hurwitz numbers is derived in~\cite{DDMTop}. Its spectral curve data corresponds to the differential $\omega_{0,1}=-Y\, dX$ with
$$
X(Z)=\frac{Z-1}{Z^2},\quad Y(Z)=-Z,\qquad X_i=X(Z).
$$

Here $Z$ is the global affine coordinate on the spectral curve $\Sigma=\mathbb C P^1$ and $X$ is a local coordinate at $Z=\infty$ that extends as a degree~$2$ global rational function on the spectral curve. Applying an affine change $Z=\frac{2}{1-z}$ of the coordinate on the spectral curve we can rewrite its equation in an equivalent form
$$
x(z)=\frac12-2X=\frac{z^2}{2},\quad y(z)=-\frac12Y=\frac{1}{1-z}.
$$
Both spectral curves produce the same topological recursion since $-Y\,dX=-y\,dx$.
The expansion point $Z=\infty$ corresponds to $z=1$ and the function $X$ takes in the $z$ coordinate the form $X=\frac14-\frac12 x=\frac{1-z^2}{4}$ with the inverse change $z=\sqrt{1-4 X}$.
Besides, adding an even summand to the $y$-function does not affect the differentials $\omega_{g,n}$ of topological recursion for $(g,n)\ne(0,1)$. Therefore, we can replace the $y$-function in the above equation of the spectral curve by its odd part with respect to the involution $z\mapsto -z$ which is
$$
\frac12\Bigl(\frac1{1-z}-\frac{1}{1+z}\Bigr)=\frac{z}{1-z^2}.
$$
This proves an equivalence of the two forms of topological recursion for the monotone Hurwitz numbers.
\end{proof}
In recent work \cite{BDKSExp}, the structure of partition functions for a large class of Hurwitz numbers, which includes simple and monotone Hurwitz numbers, has been studied.  One property of this structure is the relationship of these Hurwitz numbers to topological recursion.

The proposition implies that the monotone Hurwitz numbers can be expressed as linear combinations of the intersection numbers involving $\psi$ and $\kappa$ classes. More explicitly, it is easy to compute that the form $\frac{(2k+1)!!dz}{z^{2k+2}}$ is given in the local coordinate $X$ at $z=1$ by
$$
\frac{(2k+1)!!dz}{z^{2k+2}}=\sum_{i=1}^\infty(2i+1)(2i+3)\dots(2(i+k)-1)\binom{2i}{i}i\,X^{i-1}dX.
$$
Substituting into~\eqref{intnum} and selecting the coefficient of a particular monomial in $X$-variables we arrive at the following known analog of the ELSV formula, see~\cite{ALSRam,DKaMon,ELSVHur}. Denote
$$
\Phi_k(\psi)=1+(2k+1)\psi+(2k+1)(2k+3)\psi^2+\dots=\sum_{i=0}^\infty\left(\prod_{j=1}^i(2(j+k)-1)\right)\psi^i
$$
where we use the convention that the empty product is 1.
\begin{corollary}
The monotone Hurwitz numbers in the case $2g-2+n>0$ can be computed as the following intersection numbers
\begin{align*}
\vec h_{g,(k_1,\dots,k_n)}&=\prod_{i=1}^n\!\!\binom{2k_i}{k_i}~\vec P_{g,n}(k_1,\dots,k_n),\\
\vec P_{g,n}(k_1,\dots,k_n)&=\int_{\overline{\modm}_{g,n}}
\bk^*\cdot\Phi_{k_1}(\psi_1)\dots\Phi_{k_n}(\psi_n).
\end{align*}
\end{corollary}
\noindent Note that, by definition, $\vec P_{g,n}$ is a symmetric polynomial in $k_1,\dots,k_n$.  A consequence of Theorem~\ref{numeric} is that for $g>1$, 
\[\vec P_{g,n}(-\tfrac12,-\tfrac12,\dots,-\tfrac12)=0
\]
since $\Phi_{-\frac12}(\psi)=1$ and $\deg\bk^*<3g-3+n$ weakly.  More generally, 
\[\deg\Phi_{-\frac12(2m+1)}(\psi)=m\]
hence by Theorem~\ref{numeric} the following vanishing properties of the polynomials $\vec P_{g,n}$ evaluated on negative half-integers holds.
\begin{corollary}
For $m_i\in\bn$,
\[ \sum_{i=1}^n m_i<g-1\quad\Rightarrow\quad \vec P_{g,n}\big(-\tfrac12(2m_1+1),\dots,-\tfrac12(2m_n+1)\big)=0.
\]
\end{corollary}
By direct calculation, via the Witten-Kontsevich Theorem~\ref{KW} or with admcycles, one can verify:
\[\int_{\overline{\modm}_{g}}K_{3g-3}=(-1)^g\frac{B(2g)}{2g(2g-2)}=|\chi(\modm_g)|,\quad 2\leq g\leq 9\]
where $\chi(\modm_g)$ is the orbifold Euler characteristic of the moduli space of curves.
Using the spectral curve and recursions between the 1-point correlators proven in \cite{CDoGen} it is likely that we can prove this formula for all $g\geq 2$, using a method analogous to the calculation of the orbifold Euler characteristic of the moduli space of curves by Harer and Zagier in \cite{HZaEul} from their recursions between the 1-point correlators.



\begin{thebibliography}{99}

\bibitem{AJaMod} Abramovich, D. and Jarvis, T.
\emph{Moduli of twisted spin curves.}
Proc. Amer. Math. Soc. {\bf 131} (2003), 685-699.

\bibitem{ALSRam}
Alexandrov, A., Lewanski, D. and  Shadrin, S.,  
\emph{Ramifications of Hurwitz theory,
KP integrability and quantum curves.} 
J. High Energy Phys. {\bf 5} (2016).


\bibitem{ACoCom} Arbarello, E. and Cornalba, M.
\emph{Combinatorial and algebro-geometric cohomology classes on the moduli spaces of curves.}
 J. Algebraic Geom. {\bf 5} (1996), 705-749.



\bibitem{BGrExt} Br\'{e}zin, E. and Gross, D.J.
\emph{The external field problem in the large $N$ limit of QCD.}
 Phys. Lett. B97 (1980) 120.
 
 \bibitem{BDKSExp} Bychkov, B., Dunin-Barkowski, P., Kazarian, M. and Shadrin, S.
\emph{Explicit closed algebraic formulas for Orlov-Scherbin n-point functions.}
\href{http://arxiv.org/abs/2008.13123}{arXiv:2008.13123}


\bibitem{CDoGen}
Chaudhuri, A. and Do, N.
\emph{Generalisations of the Harer-Zagier recursion for 1-point functions.} 
J. Algebraic Comb {\bf 53} (2021), 1-35.

\bibitem{CNoTop} Chekhov, L. and Norbury, P.
\emph{Topological recursion with hard edges.}
Inter. Journal Math. {\bf 30} (2019) 1950014 (29 pages).


\bibitem{DSZadm} Delecroix, V.; Schmitt, J. and  van Zelm, J.
\emph{admcycles -- a Sage package for calculations in the tautological ring of the moduli space of stable curves.}
\href{http://arxiv.org/abs/2002.01709}{arXiv:2002.01709}

\bibitem{DNoTop}
Do, N. and Norbury, P.
\emph{Topological recursion on the Bessel curve.}
{\em Commun. Number Theory Phys.} {\bf 12} (2018), 53-73.


\bibitem{DDMTop} Do, N., Dyer, A. and  Mathews, D. 
\emph{Topological recursion and a quantum curve for monotone Hurwitz numbers.} 
\href{http://arxiv.org/abs/1408.3992}{arXiv:1408.3992}

\bibitem{DKaMon}
Do, N.,  Karev, M., 
\emph{Monotone orbifold Hurwitz numbers.} 
Zapiski Nauchnykh Seminarov POMI, {\bf 446}, (2016).
J. Math. Sci. {\bf 226}, (2017).






\bibitem{ELSVHur}  Ekedahl, T., Lando, S., Shapiro, M. and Vainshtein, A.
\emph{Hurwitz numbers and intersections on moduli spaces of curves.}
Invent. Math. {\bf 146} (2001), 297-327.

\bibitem{EynInv} Eynard, B.
\emph{Invariants of spectral curves and intersection theory of moduli spaces of complex curves.}
Commun. Number Theory Phys. {\bf 8} (3), (2014), 541--588.

\bibitem{EOrInv}
Eynard, B. and Orantin, N.
\emph{Invariants of algebraic curves and topological expansion.}
Commun. Number Theory Phys. {\bf 1} (2007), 347-452.

\bibitem{EOrTop} Eynard, B. and Orantin, N.
\emph{Topological recursion in enumerative geometry and random matrices.}
J. Phys. A {\bf 42}, no. 29, 293001, 117 pp., (2009).

\bibitem{FPaCon} Faber, C. and Pandharipande, R.
\emph{Constructions of nontautological classes on moduli spaces of curves.}
Michigan Math. J. {\bf 51} (2003), 93-110.

\bibitem{FPaLog} Faber, C. and Pandharipande, R.
\emph{Logarithmic series and Hodge integrals in the tautological ring. With an appendix by Don Zagier.}
Michigan Math. J. {\bf 48} (2000), 215-252.

\bibitem{FPaRel} Faber, C. and Pandharipande, R.
\emph{Relative maps and tautological classes.}
J. Eur. Math. Soc. {\bf 7} (2005), 13-49.


\bibitem{GNeUni}
 Gross, D. J.,  Newman,  M. J., 
 \emph{Unitary and Hermitian matrices in an external field. 2:
The Kontsevich model and continuum Virasoro constraints.} 
Nucl. Phys. B {\bf 380} (1992) 168.

\bibitem{GWiPos} Gross, D. and Witten, E.
\emph{Possible third-order phase transition in the large-$N$ lattice gauge theory.}
Phys. Rev. D21 (1980) 446.

\bibitem{HZaEul} Harer, J. and Zagier, D.
\emph{The Euler characteristic of the moduli space of curves.}
Invent. Math. {\bf 85} (1986), 457-485.

\bibitem{IZuCom} Itzykson, C. and Zuber, J.-B.
\emph{Combinatorics of the Modular Group II, the Kontsevich Integrals.}
Int. Journ. Mod. Phys. A {\bf 7} (1992), 5661-5705.

\bibitem{JKVMod} Jarvis, T.; Kimura, T. and Vaintrob, A.
\emph{Moduli Spaces of Higher Spin Curves and Integrable Hierarchies.}
Compositio Mathematica {\bf 126} (2001), 157-212.

\bibitem{KonInt} Kontsevich, M.
\emph{Intersection theory on the moduli space of curves and the matrix {A}iry function.}
Comm. Math. Phys. {\bf 147} (1992), 1-23.

\bibitem{LPSZChi} Lewanski, D.; Popolitov, A.; Shadrin, S. and Zvonkine, D.
\emph{Chiodo formulas for the r-th roots and topological recursion.}
Lett. Math. Phys. {\bf 107} (2017), 901-919.

\bibitem{MZoInv}  Manin, Y.I.; Zograf, P.
\emph{Invertible cohomological field theories and Weil-Petersson volumes.}
Ann. Inst. Fourier (Grenoble) {\bf 50} (2000), 519-535.

\bibitem{MJDSol} Miwa, T.; Jimbo, M.; Date, E.
\emph{Solitons. Differential equations, symmetries and infinite-dimensional algebras.}
Cambridge Tracts in Mathematics, {\bf 135} Cambridge University Press, Cambridge, 2000.

\bibitem{NatMod}
Natanzon, S.
\emph{Moduli of Riemann Surfaces, Real Algebraic Curves, and Their Superanalogs.}
Transl. Math. Monographs {\bf 225}, AMS, 2004.

\bibitem{NorEnu}
Norbury, P.
\emph{Enumerative geometry via the moduli space of super Riemann surfaces.}
\href{http://arxiv.org/abs/2005.04378}{arXiv:2005.04378}

\bibitem{NorNew}
Norbury, P.
\emph{A new cohomology class on the moduli space of curves.}
To appear {\em Geom. Topol.}

\bibitem{PanCal}
Pandharipande, R.
\emph{A calculus for the moduli space of curves.}
Proc. Symp. Pure Math. {\bf 97} (2018), 459-487.

\bibitem{PanKap}
Pandharipande, R.
\emph{The $\kappa$ ring of the moduli of curves of compact type.}
Acta Math. {\bf 208} (2012), 335-388.


\bibitem{PPZRel} Pandharipande, R.; Pixton, A. and Zvonkine, D.
\emph{Relations on $\overline{\modm}_{g,n}$ via 3-spin structures.}
J. Amer. Math. Soc. {\bf 28} (2015), 279-309.

\bibitem{PixTau} Pixton, A.
\emph{The tautological ring of the moduli space of curves.}
PhD thesis, Princeton University 2013.

\bibitem{SWiJTG}
Stanford, D. and Witten, E.
\emph{JT Gravity and the Ensembles of Random Matrix Theory}
\href{http://arxiv.org/abs/1907.03363}{arXiv:1907.03363}

\bibitem{TelStr}
Teleman, C.
\emph{The structure of 2D semi-simple field theories},
Invent. Math. {\bf 188} (2012), 525-588.

\bibitem{WitTwo} Witten, E.
\emph{Two-dimensional gravity and intersection theory on moduli space.} Surveys in differential geometry (Cambridge, MA, 1990), 243--310, Lehigh Univ., Bethlehem, PA, 1991.

\end{thebibliography}
\end{document}